\documentclass[12pt]{amsart}

\newtheorem{theorem}{Theorem}[section]
\newtheorem{lemma}[theorem]{Lemma}
\newtheorem{proposition}[theorem]{Proposition}
\newtheorem{corollary}[theorem]{Corollary}   
\newtheorem{definition}[theorem]{Definition}
\newtheorem{example}[theorem]{Example}

\newtheorem{question}[theorem]{Question}

\numberwithin{equation}{section}

\usepackage{times}
\usepackage{enumerate}
\usepackage{mathrsfs}
\usepackage{tfrupee}
\usepackage{tikz}
\usetikzlibrary{chains,fit}
\usetikzlibrary{shapes,snakes}
\usetikzlibrary{graphs}
\usepackage{graphicx,adjustbox}
\usepackage{hyperref}
\usepackage{amsmath,amssymb}
\usepackage{amscd}
\usepackage{graphicx}
\usepackage[all]{xy}

\begin{document}

\title{The $\mathrm{v}$-number of Monomial Ideals}
\author{
Kamalesh Saha \and Indranath Sengupta
}
\date{}

\address{\small \rm  Discipline of Mathematics, IIT Gandhinagar, Palaj, Gandhinagar, 
Gujarat 382355, INDIA.}
\email{kamalesh.saha@iitgn.ac.in}

\address{\small \rm  Discipline of Mathematics, IIT Gandhinagar, Palaj, Gandhinagar, 
Gujarat 382355, INDIA.}
\email{indranathsg@iitgn.ac.in}
\thanks{The second author is the corresponding author; supported by the 
MATRICS research grant MTR/2018/000420, sponsored by the SERB, Government of India.}

\date{}



\allowdisplaybreaks

\begin{abstract}
We show that the $\mathrm{v}$-number of an arbitrary monomial 
ideal is bounded below by the $\mathrm{v}$-number of its polarization and also, find a criteria for the equality. By showing the additivity of associated primes of monomial ideals, we obtain v-number is additive for arbitrary monomial ideals. We prove that the $\mathrm{v}$-number 
$\mathrm{v}(I(G))$ of the edge ideal $I(G)$, the induced 
matching number $\mathrm{im}(G)$ and the regularity 
$\mathrm{reg}(R/I(G))$ of a graph $G$, satisfy 
$\mathrm{v}(I(G))\leq \mathrm{im}(G)\leq \mathrm{reg}(R/I(G))$, 
where $G$ is either a bipartite graph, or a $(C_{4},C_{5})$-free 
vertex decomposable graph, or a whisker graph. There is an open problem 
in \cite{v}, whether $\mathrm{v}(I)\leq \mathrm{reg}(R/I)+1$ for any 
square-free monomial ideal $I$. We show that $\mathrm{v}(I(G))>\mathrm{reg}(R/I(G))+1$, 
for a disconnected graph $G$. We derive some inequalities of 
$\mathrm{v}$-numbers which may be helpful to answer the above problem 
for the case of connected graphs. We connect $\mathrm{v}(I(G))$ with 
an invariant of the line graph $L(G)$ of $G$. For a simple connected 
graph $G$, we show that $\mathrm{reg}(R/I(G))$ can be arbitrarily 
larger than $\mathrm{v}(I(G))$. Also, we try to see how the 
$\mathrm{v}$-number is related to the Cohen-Macaulay property 
of square-free monomial ideals.
\end{abstract}

\maketitle

\section{Introduction}
Let $R=K[x_{1}\,\ldots, x_{n}]=\bigoplus_{d=0}^{\infty}R_{d}$ 
denote the polynomial ring in $n$ variables over a field $K$, 
with the standard gradation. Given a graph $G$, we assume 
$V(G)=\{x_{1},\ldots,x_{n}\}$ and all graphs are assumed to be 
simple graphs.
\medskip

For a graded ideal $I$ of $R$, the set of associated prime ideals of $I$, 
denoted by $\mathrm{Ass}(I)$ or $\mathrm{Ass}(R/I)$, is the collection of 
prime ideals of $R$ of the form $(I:f)$, for some $f\in R_{d}$. A prime ideal $\mathfrak{p}\in\mathrm{Ass}(R/I)$ is said to be a \textit{minimal prime} of $I$ if for all $\mathfrak{q}\in \mathrm{Ass}(R/I)$ with $\mathfrak{p}\neq \mathfrak{q}$ we have $\mathfrak{q}\not\subset\mathfrak{p}$. If an associated prime ideal of $I$ is not minimal, then $I$ is called an \textit{embedded prime} of $I$.

\begin{definition}[\cite{cstvv}, Definition 4.1]\label{v1.1}{\rm
Let $I$ be a proper graded ideal of $R$. Then v-\textit{number} of $I$ is denoted 
by $\mathrm{v}(I)$ and is defined by
$$ \mathrm{v}(I):=
 \mathrm{min}\{d\geq 0 \mid \exists\, f\in R_{d}\,\,\text{and}\,\, \mathfrak{p}\in \mathrm{Ass}(I)\,\, \text{with}\,\, (I:f)=\mathfrak{p} \}.$$
 }
\end{definition}

\noindent For each $\mathfrak{p}\in \mathrm{Ass}(I)$, we can locally define v-number as
$$\mathrm{v}_{\mathfrak{p}}(I):=\mathrm{min}\{d\geq 0 \mid \exists\, f\in R_{d}\,\, \text{with}\,\, (I:f)=\mathfrak{p} \}.$$
Then $\mathrm{v}(I)=\mathrm{min}\{\mathrm{v}_{\mathfrak{p}}(I)\mid \mathfrak{p}\in\mathrm{Ass}(I)\}$.
\medskip

The $\mathrm{v}$-number of $I$ was introduced as an invariant of the graded 
ideal $I$, in \cite{cstvv}, in the study of Reed-Muller-type codes. This invariant of $I$ helps us understand the behaviour 
of the generalized minimum distance function $\delta_{I}$ of $I$, 
in the said context. See \cite{cstvv}, \cite{v}, 
\cite{mpv}, \cite{npv}, for further details on this.
\medskip

Procedure A1 in \cite{grv} helps us compute the 
$\mathrm{v}$-number of monomial ideals using 
\textit{Macaulay2} \cite{mac2}. In \cite{v}, 
Jaramillo and Villarreal have discussed some properties 
of $\mathrm{v}(I)$ and have proved combinatorial formula 
of $\mathrm{v}(I)$, where $I$ is a square-free monomial ideals. They 
have proved that $\mathrm{v}(I)\leq \mathrm{reg}(R/I)$ is 
satisfied for several cases of square-free monomial ideals $I$. 
In the same article, the authors have also disproved 
Conjecture 4.2 (\cite{npv}) by giving an example (\cite{v}, Example 5.4) 
of a connected graph $G$, with $3=\mathrm{v}(I(G))>\mathrm{reg}(R/I(G))=2$. 
They have proposed an open problem in \cite{v}, whether 
$\mathrm{v}(I)\leq \mathrm{reg}(R/I)+1$, for any square-free 
monomial ideal $I$. In this paper, we give a counter 
example (Example \ref{v5.1}) to this open problem and 
modify the question (Question \ref{v5.2}) for edge ideals 
of those clutters which can not be written as disjoint union 
of two clutters. We try to give partial answer to this question. We find relation between the $\mathrm{v}$-number of an arbitrary 
monomial ideal and the v-number of its polarization along with some criteria for equality.
\medskip

Bounds of Castelnuovo-Mumford regularity of edge ideals (see \cite{bc}, \cite{bcreg}, \cite{dhs}, \cite{fht}, \cite{hahui}, \cite{katz}, \cite{w}) 
and bounds of induced matching number of graphs (see \cite{cam}, 
\cite{camimint}, \cite{camst}, \cite{joos}, \cite{marin}, \cite{zito}) 
are two trending topics in the research of commutative algebra and 
combinatorics respectively. Also, obtaining induced matching number 
in general is $NP$-hard. So it would be an interesting problem to 
find the bounds of regularity and induced matching number by the 
$\mathrm{v}$-number. Considering $I(G)$ as the edge ideal of a 
graph $G$, we give a relation between $\mathrm{v}(I)$, 
$\mathrm{reg}(R/I)$ and $\mathrm{im}(G)$ for bipartite graphs 
(Theorem \ref{imbvr}) , $(C_{4},C_{5})$-free vertex-decomposable graphs 
(Theorem \ref{v4.9}), whisker graphs (Theorem \ref{v4.10}) etc. 
We also obtain some results on $\mathrm{v}$-number and propose some problems. 
The paper is arranged in the following manner.
\medskip

In Section 2, we discuss the Preliminaries. We recall some definitions, 
notations, basic concepts pertinent to Graph theory and Commutative 
Algebra and results from \cite{v}. In Section 3, 
our main result is the following:
\medskip

\noindent \textbf{Theorem \ref{vpol}.} Let $I$ be a monomial ideal. If there exists 
$$\mathfrak{p}=\big< x_{s_{1},b_{s_{1}}},\ldots, x_{s_{k},b_{s_{k}}}\big> \in \mathrm{Ass}(I(\mathrm{pol}))$$ such that $\mathrm{v}(I(\mathrm{pol}))=\mathrm{v}_{\mathfrak{p}}(I(\mathrm{pol}))$ and there is no embedded prime of $I$ properly containing $\big< x_{s_{1}},\ldots, x_{s_{k}}\big>$, then $$\mathrm{v}(I)=\mathrm{v}(I(\mathrm{pol})).$$

\noindent In general, we get $\mathrm{v}(I(\mathrm{pol}))\leq \mathrm{v}(I)$ (Corollary \ref{vcor3.5}). Also, in this section, we generalize some results proved 
in \cite{v}, related to v-number of square-free monomial ideals to 
arbitrary monomial ideals with special type which includes the monomial ideals having no embedded primes. We show the additive property of associated primes of monomial ideals (Lemma \ref{vassp}) and using this result we show that the additivity holds for v-numbers of monomial ideals i.e.,
\medskip

\noindent \textbf{Proposition \ref{v3.5}.}
Let $I_{1}\subset R_{1}=K[\mathbf{x}]$ and $I_{2}\subset R_{2}=K[\mathbf{y}]$ be two monomial ideals and consider $R=K[\mathbf{x,y}]$. Then we have
$$ \mathrm{v}(I_{1}R + I_{2}R) = \mathrm{v}(I_{1}) + \mathrm{v}(I_{2}).$$

\noindent In addition, we derive some properties 
(see Proposition \ref{v3.8}) of $\mathrm{v}(I)$ for any square-free 
monomial ideal $I$. For any graph $G$ we show that 
$\mathrm{v}(I(G))\leq \alpha_{0}(G)$ (Proposition \ref{v3.9}), 
where $\alpha_{0}(G)$ is the vertex covering number of $G$. 
In section 4, we relate $\mathrm{v}(I(G))$ with an invariant of 
$L(G)$, the line graph of $G$ (see Proposition \ref{v4.1}). We 
derive some properties (Proposition \ref{v4.2}) of $\mathrm{v}$-number 
of graphs (see Proposition \ref{v4.2}), which could be helpful to find the relation between 
the $\mathrm{v}$-number and the regularity of edge ideals. We find 
the following relations between 
$\mathrm{v}(I(G)),\,\, \mathrm{reg}(R/I(G)),\,\, \mathrm{im}(G)$ for certain classes of graphs $G$:
\medskip

\noindent\textbf{Theorems \ref{imbvr}, \ref{v4.9}, \ref{v4.10}.}
If $G$ is a bipartite graph or $(C_{4},C_{5})$-free vertex decomposable 
graph or whisker graph, then
 $$\mathrm{v}(I(G))\leq\mathrm{im}(G)\leq \mathrm{reg}(R/I(G)).$$
Also, we show that for a graph $G$, the difference between 
$\mathrm{v}(I(G))$ and $\mathrm{reg}(R/I(G))$ may be arbitrarily 
large (see Corollary \ref{v3.10}). In Proposition \ref{v4.18}, 
we try to relate Cohen-Macaulay property of $(R/I)$ with 
$\mathrm{v}(I^{\vee})$, where $I=I(\mathcal{C})$ is an 
edge ideal of a clutter $\mathcal{C}$, such that $\mathcal{C}$ 
can not be written as union of two disjoint clutters and 
$I^{\vee}$ denotes the Alexander dual ideal (Definition \ref{v4.17}) 
of $I$. In section 5, we give a counter example (Example \ref{v5.1}) 
to the problem given in \cite{v} and modify the question 
(see Question \ref{v5.2}). We also propose some open problems 
related to the $\mathrm{v}$-number in terms of regularity, depth 
and induced matching number.

\section{Preliminaries}
In this section, we will recall some basic definitions, results and notations of graph theory and commutative algebra. Also we mention some results, concepts and notations from \cite{v}.
\medskip

In $R$ we denote a monomial $x_{1}^{a_{1}}\cdots x_{n}^{a_{n}}$ by $\mathbf{x^{a}}$, where $\mathbf{a}=(a_{1},\ldots,a_{n})\in \mathbb{N}^{n}$ and $\mathbb{N}$ denotes the set of all non-negative integers. An ideal $I\subset R$ is called a monomial ideal if it is minimally generated by a set of monomials in $R$. The set of minimal monomial generators of $I$ is unique and it is denoted by $G(I)$. If $G(I)$ consists of only square-free monomials, then we say $I$ is a square-free monomial ideal.

\begin{definition}{\rm
A \textit{clutter} $\mathcal{C}$ is a pair of two sets $(V(\mathcal{C}),E(\mathcal{C}))$, where $V(\mathcal{C})$ is called the vertex set and $E(\mathcal{C})$ is a collection of subsets of $V(\mathcal{C})$, called edge set, such that no two elements (called edges) of $E(\mathcal{C})$ contains each other. A clutter is also known as \textit{simple hypergraph}. A simple graph is an example of a clutter, whose edges are of cardinality two.
} 
\end{definition}

Let $\mathcal{C}$ be a clutter on a vertex set $V(\mathcal{C})$. An edge $e\in E(\mathcal{C})$ is said to be \textit{incident} on a vertex $x\in V(\mathcal{C})$ if $x\in e$. A subset $C\subset V(\mathcal{C})$ is called a \textit{vertex cover} of $\mathcal{C}$ if any $e\in E(\mathcal{C})$ is incident to a vertex of $C$. If a vertex cover is minimal with respect to inclusion, then we call it a \textit{minimal vertex cover}. The cardinality of a minimum (smallest) vertex cover is known as the \textit{vertex covering number} of $\mathcal{C}$ and is denoted by $\alpha_{0}(\mathcal{C})$. Also a subset $A\subset V(\mathcal{C})$ is said to be \textit{stable} or \textit{independent} if $e\not\subset A$ for any $e\in E(\mathcal{C})$ and $A$ is said to be \textit{maximal independent set} if it is maximal with respect to inclusion. The number of vertices in a maximum (largest) independent set, denoted by $\beta_{0}(\mathcal{C})$, is called the \textit{independence number} of $\mathcal{C}$. Note that a vertex cover $C$ is a minimal vertex cover 
of $\mathcal{C}$ if and only if its complement $V(\mathcal{C})\setminus C$ is a maximal independent set.
\medskip

Let $\mathcal{C}$ be a clutter on the vertex set $V(\mathcal{C})=\{x_{1},\ldots,x_{n}\}$. Then for $A\subset V(\mathcal{C})$, we consider $X_{A}:=\prod_{x_{i}\in A} x_{i}$ as a square-free monomial in the polynomial ring $R=K[x_{1},\ldots,x_{n}]$ over a field $K$. The edge ideal of the clutter $\mathcal{C}$, denoted by $I(\mathcal{C})$, is the ideal in $R$ defined by
$$I(\mathcal{C})=\big<X_{e}\mid e\in E(\mathcal{C})\big>.$$
Set of square-free monomial ideals are in one to one correspondence with the set of clutters. For a simple graph $G$, the edge ideal $I(G)$ is generated by square-free quadratic monomials. It is a well known fact that
$$\mathrm{ht}(I(\mathcal{C}))=\alpha_{0}(\mathcal{C})\,\, \text{and}\,\, \mathrm{dim}(R/I(\mathcal{C}))=\beta_{0}(\mathcal{C}),$$
where $\mathrm{ht}(I(\mathcal{C}))$ is height of $I(\mathcal{C})$ and $\mathrm{dim}(R/I(\mathcal{C}))$ is the Krull dimension of $R/I(\mathcal{C})$. Note that $\alpha_{0}(\mathcal{C})+\beta_{0}(\mathcal{C})=n$.
\medskip

Let $A$ be a stable set of a clutter $\mathcal{C}$. Then the \textit{neighbor} set of $A$ in $\mathcal{C}$, denoted by $\mathcal{N}_{\mathcal{C}}(A)$, is defined by
$$\mathcal{N}_{\mathcal{C}}(A)=\{x_{i}\in V(\mathcal{C})\mid \{x_{i}\}\cup A\,\, \text{contains\,\,an\,\,edge\,\,of}\,\, \mathcal{C}\}.$$
We denote $\mathcal{N}_{\mathcal{C}}[A]:=\mathcal{N}_{\mathcal{C}}(A)\cup A$. Now we will recall some notations and results from \cite{v}. Let $\mathcal{F}_{\mathcal{C}}$ denotes the collection of all maximal stable sets of $\mathcal{C}$ and $\mathcal{A}_{\mathcal{C}}$ denotes the collection of those stable sets $A$ of $\mathcal{C}$ such that $\mathcal{N}_{\mathcal{C}}(A)$ is a minimal vertex cover of $\mathcal{C}$. The following theorems in \cite{v} gives the combinatorial formula for $\mathrm{v}(I(\mathcal{C}))$.

\begin{lemma}[\cite{v}, Lemma 3.4]\label{v2.2}
 Let $I=I(\mathcal{C})$ be the edge ideal of a clutter $\mathcal{C}$. Then the following hold:
 \begin{enumerate}
\item[$\mathrm{(a)}$] For $A\in \mathcal{A}_{\mathcal{C}}$, we have $(I: X_{A}) =\big<\mathcal{N}_{\mathcal{C}}(A)\big>$.
\item[$\mathrm{(b)}$] If $A$ is stable and $\mathcal{N}_{\mathcal{C}}(A)$ is a vertex cover, then $\mathcal{N}_{\mathcal{C}}(A)$ is a minimal vertex cover.
\item[$\mathrm{(c)}$] If $(I:f) =\mathfrak{p}$ for some $f\in R_{d}$ and some $\mathfrak{p}\in \mathrm{Ass}(I)$, then there is $A\in \mathcal{A}_{\mathcal{C}}$ with $\vert A\vert\leq d$ such that $(I:X_{A}) =\big<\mathcal{N}_{\mathcal{C}}(A)\big>=\mathfrak{p}$.
\item[$\mathrm{(d)}$] If $A\in\mathcal{F}_{\mathcal{C}}$, then $\mathcal{N}_{\mathcal{C}}(A) =V(C)\setminus A$ and $(I:X_{A}) =\big<\mathcal{N}_{\mathcal{C}}(A)\big>$.
 \end{enumerate}
\end{lemma}

\begin{theorem}[\cite{v}, Theorem 3.5]\label{v2.3}
Let $I=I(\mathcal{C})$ be the edge ideal of a clutter $\mathcal{C}$. If $I$ is not prime, then $\mathcal{F}_{\mathcal{C}}\subset \mathcal{A}_{\mathcal{C}}$ and
$$ \mathrm{v}(I)=\mathrm{min}\{\vert A\vert : A\in\mathcal{A}_{\mathcal{C}} \}.$$
\end{theorem}

\noindent In this paper, we use Lemma \ref{v2.2} and Theorem \ref{v2.3} frequently.
\medskip
 
Let $V = \lbrace x_{1},\ldots, x_{n}\rbrace$. A \textit{simplicial complex} $\Delta$ on the 
vertex set $V$ is a collection
of subsets of $V$, with the following properties:
\begin{enumerate}
\item[(i)] $\lbrace x_{i}\rbrace\in \Delta$ for all $x_{i}\in V$;
\item[(ii)] $F\in \Delta$ and $G\subseteq F$ imply $G\in \Delta$.
\end{enumerate}
\medskip

\noindent An element $F\in \Delta$ is called a \textit{face} of $\Delta$. A maximal face 
of $\Delta$ is called a \textit{facet} of 
$\Delta$. For a vertex $v\in V$, $\mathrm{del}_{\Delta}(v)$ is a subcomplex, 
called \textit{deletion} of $v$, on the vertex set $V\setminus \{v\}$ given by
$$\mathrm{del}_{\Delta}(v):=\{F\in \Delta\mid v\not\in F\}$$
and the $\mathrm{lk}_{\Delta}(v)$, called \textit{link} of $v$, is the 
subcomplex of $\mathrm{del}_{\Delta}(v)$ given by
$$\mathrm{lk}_{\Delta}(v):=\{F\in \Delta\mid v\not\in F\,\,\text{and}\,\, F\cup\{v\}\in \Delta\}.$$
If $V$ is the only facet of $\Delta$, then $\Delta$ is called a \textit{simplex}. 
\begin{definition}{\rm
A simplicial complex $\Delta$ is called \textit{vertex decomposable} if either $\Delta$ is
a simplex, or $\Delta=\phi$, or $\Delta$ contains a vertex $v$ such that
\begin{enumerate}[(a)]
\item both of $\mathrm{del}_{\Delta}(v)$ and $\mathrm{lk}_{\Delta}(v)$ are vertex decomposable, 

\noindent and
\item every facet of $\mathrm{del}_{\Delta}(v)$ is a facet of $\Delta$.
\end{enumerate}
A vertex $v$ satisfying condition (2) is called a \textit{shedding} vertex of $\Delta$.
}
\end{definition}

\noindent The \textit{independence complex} $\Delta_{\mathcal{C}}$ of a clutter $\mathcal{C}$ is a simplicial complex whose faces are the stable sets of $\mathcal{C}$. Note that the Stanley-Reisner ideal $I_{\Delta_{\mathcal{C}}}$ is equal to $I(\mathcal{C})$.
\medskip

\begin{definition}{\rm
Let $I$ and $J$ be ideals of a ring $R$. The \textit{colon ideal} of $I$ with respect to $J$ is an ideal of $R$, denoted by $(I:J)$ and is defined as
$$(I:J)= \{u\in R \mid uv\in I\,\,\text{for\,\, all}\,\,v\in J\}.$$
For an element $f\in R$, $(I:f):=(I:(f))$. If $I$ is a monomial ideal and $f\in R$ is a monomial, then by (\cite{hhmon}, Proposition 1.2.2) we have
$$ (I:f)=\big <\dfrac{u}{\mathrm{gcd}(u,f)}\mid u\in G(I)\big>.$$
}
\end{definition}

Let $I$ be an ideal in a ring $R$. Then a presentation $I=\bigcap_{i=1}^{k} \mathfrak{q}_{i}$, where each $\mathfrak{q}_{i}$ is a primary ideal, is called a primary decomposition of $I$. A primary decomposition is irredundant if no $\mathfrak{q}_{i}$ can be omitted in the presentation and $\mathfrak{p}_{i}\not =\mathfrak{p}_{j}$ for $i\not=j$, where $\mathfrak{p}_{i}=\sqrt{\mathfrak{q}_{i}}$. Each $\mathfrak{p}_{i}$ is said to be an associated prime ideal of $I$ and the set of associated prime ideals of $I$ is denoted by $\mathrm{Ass}(I)$ or $\mathrm{Ass}(R/I)$. From (\cite{am}, Theorem 4.5) and (\cite{hhmon}, Corollary 1.3.10), we can say that the associated prime ideals of a monomial ideal $I$ are precisely the prime ideals of the form $(I:f)$ for some monomial $f\in R$. If a monomial ideal can not be written as proper intersection of two other monomial ideals, then we say it is irreducible. For a monomial ideal $I$ , a presentation of the form $I=\bigcap_{i=1}^{k} \mathfrak{q}_{i}$, where each $\mathfrak{q}_{i}$ is irreducible, is called an irredundant irreducible decomposition if no $\mathfrak{q}_{i}$ can be omitted in the decomposition.
By (\cite{hhmon}, Theorem 1.3.1 and Corollary 1.3.2), any monomial ideal can be written as an unique irredundant intersection of irreducible monomial ideals and the irreducible components are precisely generated by pure powers of the variables.
\medskip

\begin{lemma}[\cite{vil}, Lemma 6.3.37]
Let $I=I(\mathcal{C})$ be an edge ideal of a clutter $\mathcal{C}$. Then $\mathfrak{p}\in\mathrm{Ass}(I)$ if and only if $\mathfrak{p}=\big< C\big>$ for some minimal vertex cover $C$ of $\mathcal{C}$.
\end{lemma}

\begin{definition} [\cite{peeva}, Construction 21.7]{\rm 
The \textit{polarization} of monomials of type $x_{i}^{a_{i}}$ is defined as $x_{i}^{a_{i}}(\mathrm{pol})=\prod_{j=1}^{a_{i}} x_{i,j}$ and the \textit{polarization} of $\mathbf{x^{a}}=x_{1}^{a_{1}}\cdots x_{n}^{a_{n}}$ is defined to be 
$$\mathbf{x^{a}}(\mathrm{pol})=x_{1}^{a_{1}}(\mathrm{pol})\cdots x_{n}^{a_{n}}(\mathrm{pol}).$$ 
For a monomial ideal $I=\big< \mathbf{x^{a_{1}}},\ldots,\mathbf{x^{a_{n}}}\big>\subseteq R$, the \textit{polarization} $I(\mathrm{pol})$ is defined to be the square-free monomial ideal
$$ I(\mathrm{pol})=\big< \mathbf{x^{a_{1}}}(\mathrm{pol}),\ldots,\mathbf{x^{a_{n}}}(\mathrm{pol})\big>$$
in the ring $R(\mathrm{pol})=K[x_{i,j}\mid 1\leq i\leq n,\, 1\leq j\leq r_{i}]$ , where $r_{i}$ is the power of $x_{i}$ in the lcm of $\{\mathbf{x^{a_{1}}},\ldots,\mathbf{x^{a_{n}}}\}$.
}
\end{definition}

\begin{definition}{\rm
Let $F$ be a minimal graded free resolutions of $R/I$ as $R$ module such that
{\normalsize
$$\textbf{F.}\,\,\, 0\rightarrow \bigoplus_{j} R(-j)^{\beta_{k,j}}\rightarrow \cdots \rightarrow \bigoplus_{j} R(-j)^{\beta_{1,j}}\rightarrow R\rightarrow R/I\rightarrow 0, $$}
\noindent where $I$ is a graded ideal of the graded ring $R$. The \textit{Castelnuovo-Mumford regularity} of $R/I$ (in short \textit{regularity} of $R/I$) is denoted by $\mathrm{reg}(R/I)$ and defined as
 $$ \mathrm{reg}(R/I)=\mathrm{max}\{j-i\mid \beta_{i,j}\not=0\}.$$
 The \textit{projective dimension} of $R/I$ is defined to be
 $$\mathrm{pd}(R/I)=\mathrm{max}\{i\mid \beta_{i,j}\not=0\,\, \text{for}\,\, \text{some}\,\, j\}=k.$$
 }
\end{definition}

For a clutter $\mathcal{C}$ and $A\subset V(\mathcal{C})$, we define the \textit{induced clutter}  $\mathcal{C}\setminus A$ on the vertex set $V(\mathcal{C})\setminus A$ with $E(\mathcal{C}\setminus A)=\{e\in E(\mathcal{C})\mid e\cap A=\phi\}$. If $A=\{x_{i}\}$, $\mathcal{C}\setminus \{x_{i}\}$ is the cluuter, called deletion of $x_{i}$ and in this case $\big<I(\mathcal{C}\setminus \{x_{i}\}),x_{i}\big>=\big<I(\mathcal{C}),x_{i}\big>$. We often denote the ideal generated by $I$ and $f$ by 
$(I,f)$ instead of $\big< I,f\big>$.

\begin{definition}\label{v2.9}{\rm
Let $G$ be a graph. A set $M\subset E(G)$ is said to be a \textit{matching} in $G$ if no two edges in $M$ are adjacent, i.e., no two edges in $M$ share a common vertex. A matching $M=\{e_{1},\ldots,e_{k}\}$ is called an \textit{induced matching} in $G$ if the induced subgraph on the vertex set $\bigcup_{i=1}^{k}e_{i}$ contains only $M$ as the edge set, i.e., no two edges in $M$ are joined by an edge. The cardinality of a maximum (largest) induced matching in $G$ is known as the \textit{induced matching number} of $G$ and we denote it by $\mathrm{im}(G)$.
}
\end{definition}

\section{v-Number of Monomial Ideals via Polarization}
The $\mathrm{v}$-number of square-free monomial ideals has been discussed broadly in \cite{v}. 
In this section, we study the v-number of arbitrary monomial ideals using the technique of 
polarization and generalize some results of \cite{v}.

\begin{proposition}\label{v3.1}
Let $I$ be a monomial ideal and $f=x_{1}^{a_{1}}\cdots x_{n}^{a_{n}}$ be a monomial such that $(I:f)=\big<x_{s_{1}},\ldots, x_{s_{k}}\big>$, where $a_{i}\leq$ highest power of $x_{i}$ appears in $G(I)$. Then 
$$(I(\mathrm{pol}):f(\mathrm{pol}))=\big< x_{s_{1},b_{s_{1}}},\ldots, x_{s_{k},b_{s_{k}}}\big>,$$ where $b_{s_{i}}-1$ is the power of $x_{s_{i}}$ in $f$.
\end{proposition}

\begin{proof}
We know $(I:f)=\big< \dfrac{u}{\mathrm{gcd}(u,f)}\mid u\in G(I)\big>$. Therefore $x_{s_{i}}=\dfrac{u_{i}}{\mathrm{gcd}(u_{i},f)}$, for some $u_{i}\in G(I)$. Consider the ring $R(\mathrm{pol})$ corresponding to the ideal $I(\mathrm{pol})$. By the given condition on $f$ we have $f(\mathrm{pol})\in R(\mathrm{pol})$. Let $u_{i}=x_{1}^{b_{1}}\cdots x_{n}^{b_{n}}$. Then $\mathrm{gcd}(u_{i},f)=x_{1}^{b_{1}}\cdots x_{s_{i}}^{b_{s_{i}}-1}\cdots x_{n}^{b_{n}}$ and we get 
$$ \dfrac{u_{i}(\mathrm{pol})}{\mathrm{gcd}(u_{i}(\mathrm{pol}),f(\mathrm{pol}))}=x_{s_{i},b_{s_{i}}},$$
where $b_{s_{i}}-1$ is the power of $x_{s_{i}}$ in $f$. Now suppose for some $u\in G(I)$ we have $\dfrac{u}{\mathrm{gcd}(u,f)}\in \big< x_{m}\big>$, where $m\in \{s_{1},\ldots,s_{k}\}$. Let $u=x_{1}^{r_{1}}\cdots x_{n}^{r_{n}}$ and $\mathrm{gcd}(u,f)=x_{1}^{p_{1}}\cdots x_{n}^{p_{n}}$. Then we have $r_{m}-p_{m}\geq 1$ and $r_{i}-p_{i}\geq 0$ for all $i\in[n]\setminus \{m\}$. Therefore we can write
\begin{align*}
\dfrac{u(\mathrm{pol})}{\mathrm{gcd}(u(\mathrm{pol}),f(\mathrm{pol}))} =& \dfrac{u(\mathrm{pol})}{\mathrm{gcd}(u,f)(\mathrm{pol})}\\
=&(x_{1,p_{1}+1}\cdots x_{1,r_{1}})\cdots (x_{m,p_{m}+1}\cdots x_{m,r_{m}}) \\& \cdots (x_{n,p_{n}+1}\cdots x_{n,r_{n}}).
\end{align*}
Since $r_{m}\geq p_{m}+1$, it follows that $\dfrac{u(\mathrm{pol})}{\mathrm{gcd}(u(\mathrm{pol}),f(\mathrm{pol}))} \in \big< x_{m,p_{m}+1} \big>$. Now $x_{m}^{p_{m}}\mid f$ but $x_{m}^{p_{m}+1}\nmid f$ imply $p_{m}$ is the power of $x_{m}$ in $f$. Therefore $x_{m,p_{m+1}}\in (I(\mathrm{pol}):f(\mathrm{pol}))$ and hence 
$$(I(\mathrm{pol}):f(\mathrm{pol}))=\big< x_{s_{1},b_{s_{1}}},\ldots, x_{s_{k},b_{s_{k}}}\big>,$$ where $b_{s_{i}}-1$ is the power of $x_{s_{i}}$ in $f$.
\end{proof}

\begin{lemma}\label{vlemma}
Let $I\subset R$ be a monomial ideal and $f\not\in I$ be a monomial in $R$. If $\big< x_{s_{1}},\ldots, x_{s_{k}}\big>\subseteq (I:f)$, where all $s_{i}$ are distinct, then there exists a monomial $g\in R$ such that 
$$ (I:g)=\big< x_{s_{1}},\ldots,x_{s_{r}}\big>\,\, \text{and}\,\, f\mid g,$$
for some $r\geq k$.
\end{lemma}

\begin{proof}
We know that $(I:f)=\big<\dfrac{u}{\mathrm{gcd}(u,f)}\mid u\in G(I)\big>.$ If we have $\big<x_{s_{1}},\ldots, x_{s_{k}}\big>=(I:f)$, then take $g=f$ and we are done. So we may assume $\big<x_{s_{1}},\ldots,x_{s_{k}}\big>\subsetneq (I:f)$. Then for each $1\leq i\leq k$ there exists $u_{i}\in G(I)$ such that $\dfrac{u_{i}}{\mathrm{gcd}(u_{i},f)}=x_{s_{i}}$. Let $G(I)=\{u_{1},\ldots,u_{k},u_{k+1},\ldots,u_{k+m}\}$. If $\dfrac{u_{k+1}}{\mathrm{gcd}(u_{k+1},f)}$ is divided by any of $x_{s_{1}},\ldots, x_{s_{k}}$, then $\dfrac{u_{k+1}}{\mathrm{gcd}(u_{k+1},f)}\in \big< x_{s_{1}},\ldots,x_{s_{k}}\big>$ and set $f_{1}=f$. If $\dfrac{u_{k+1}}{\mathrm{gcd}(u_{k+1},f)}=h_{1}$ is not divided by any of $x_{s_{1}},\ldots, x_{s_{k}}$, then $h_{1}$ is a non-constant monomial in $K[x_{s_{k+1}},\ldots, x_{s_{n}}]$ as $f\not\in I$. Without loss of generality we assume $x_{s_{k+1}}\mid h_{1}$ and set $f_{1}=\dfrac{fh_{1}}{x_{s_{k+1}}}$. Then $\dfrac{u_{i}}{\mathrm{gcd}(u_{i},f_{1})}=x_{s_{i}}$ is clear for each $1\leq i\leq k$. Now

\begin{align*}
 \dfrac{u_{k+1}}{\mathrm{gcd}(u_{k+1},f_{1})}&=\dfrac{u_{k+1}}{\mathrm{gcd}\big(u_{k+1},\dfrac{fh_{1}}{x_{s_{k+1}}}\big)}\\&= \dfrac{u_{k+1}}{\mathrm{gcd}(u_{k+1},f)\,\mathrm{gcd}\big(\dfrac{u_{k+1}}{\mathrm{gcd}(u_{k+1},f)},\dfrac{h_{1}}{x_{s_{k+1}}}\big)}\\ &= \dfrac{h_{1}}{\mathrm{gcd}\big(h_{1},\dfrac{h_{1}}{x_{s_{k+1}}}\big)}\\&= x_{s_{k+1}}.
 \end{align*}
 
 Therefore we get $\big< x_{s_{1}},\ldots,x_{s_{k+1}}\big>\subseteq (I:f_{1})$. Continue this process with the remaining elements of $G(I)$. Finally, we will get $f_{m}=g$ such that
 $$ (I:g)=\big< x_{s_{1}},\ldots,x_{s_{r}}\big>\,\, \text{and}\,\, f\mid g,$$
 for some $r\geq k$.
\end{proof}

\begin{proposition}\label{v3.2}
Let $I$ be a monomial ideal and consider $$\mathfrak{p}=\big< x_{s_{1},b_{s_{1}}},\ldots, x_{s_{k},b_{s_{k}}}\big> \in \mathrm{Ass}(I(\mathrm{pol}))$$ such that there exists no embedded prime of $I$ containing $\big< x_{s_{1}},\ldots, x_{s_{k}}\big>$. Let $D=\{d\mid \exists\, M\in R_{d}\,\, \text{with}\,\, (I(\mathrm{pol}):M)=\mathfrak{p}\}$. Then to find $\mathrm{min}\,D$ we can choose $M$ in such a way that $(I(\mathrm{pol}):M)=\mathfrak{p}$ and for that $M$ we will get a monomial $f$ with $\mathrm{deg}\, f\leq \mathrm{deg}\, M$ such that $(I:f)=\big< x_{s_{1}},\ldots, x_{s_{k}}\big>$.
\end{proposition}

\begin{proof}
Since $\mathfrak{p}\in \mathrm{Ass}(I(\mathrm{pol}))$, by (\cite{far}, Proposition 2.5), there exists an irredundant irreducible primary component of $I$ such that $\mathfrak{q}=\big< x_{s_{1}}^{a_{s_{1}}},\ldots, x_{s_{k}}^{a_{s_{k}}}\big>$ , where $a_{s_{i}}\geq b_{s_{i}}\geq 1$ for $i=1,\ldots,k$. Let $\mathcal{C}$ be the clutter corresponding to the ideal $I(\mathrm{pol})$ i.e., $I(\mathcal{C})=I(\mathrm{pol})$. By Lemma \ref{v2.2}, there exists a stable set $A$ of $\mathcal{C}$ such that 
$$(I(\mathrm{pol}):X_{A})=\big< \mathcal{N}_{\mathcal{C}}(A)\big>=\mathfrak{p}.$$
Now there exists $e_{i}\in E(\mathcal{C})$ such that $e_{i}\subseteq A\cup \{x_{s_{i},b_{s_{i}}}\}$ for $1\leq i\leq k$. For each $1\leq i\leq k$ we have $X_{e_{i}}=u_{i}(\mathrm{pol})$ for some $u_{i}\in G(I)$. Again $u_{i}\in \mathfrak{q}$ implies $x_{s_{j_{i}}}^{a_{s_{j_{i}}}}$ divides $u_{i}$ for some $1\leq j_{i}\leq k$. Now $x_{s_{j_{i}},b_{s_{j_{i}}}} \not\in A$ and $e_{i}\subseteq A\cup \{x_{s_{i},b_{s_{i}}}\}$ imply $s_{j_{i}}=s_{i}$. Let $c_{s_{i}}$ be the power of $x_{s_{i}}$ in $u_{i}$. Then $c_{s_{i}}\geq a_{s_{i}}$. Now consider the prime ideal $\mathfrak{p}^{\prime}= \big< x_{s_{1},a_{s_{1}}},\ldots, x_{s_{k},a_{s_{k}}}\big> \in \mathrm{Ass}(I(\mathrm{pol}))$. Let 
$$B=A\cup_{i=1}^{k}\{x_{s_{i},b_{s_{i}}}\} \setminus \cup_{i=1}^{k} \{x_{s_{i},a_{s_{i}}}\}.$$ For any $u\in G(I)$ there exists some $x_{s_{j}}^{a_{s_{j}}}$ which divides $u$, where $1\leq j\leq k$. Therefore $x_{s_{j},a_{s_{j}}} \mid u(\mathrm{pol})$, which imply corresponding edge of $u(\mathrm{pol})$ in $E(\mathcal{C})$ is not contained in $B$ and hence $B$ is a stable set in $\mathcal{C}$. Also it is clear that $$\vert A\vert =\vert B\vert\,\, \text{and}\,\, (I(\mathrm{pol}):X_{B})=\big< \mathcal{N}_{\mathcal{C}}(B)\big>=\mathfrak{p}^{\prime}.$$
Take the stable set $B^{\prime}=\cup_{i=1}^{k}(e_{i}\setminus \{x_{s_{i},a_{s_{i}}}\})\subseteq B$. Again using Lemma \ref{v2.2} we can write
$$\vert B^{\prime}\vert \leq \vert B\vert\,\, \text{and}\,\,(I(\mathrm{pol}):X_{B^{\prime}})=\big< \mathcal{N}_{\mathcal{C}}(B^{\prime})\big>=\mathfrak{p}^{\prime}.$$
Now consider the monomial $f=\mathrm{lcm}\,\{\dfrac{u_{1}}{x_{s_{1}}},\ldots, \dfrac{u_{k}}{x_{s_{k}}}\}$. Then we have $\mathrm{deg}(f)=\vert B^{\prime}\vert$ and $x_{s_{i}}f\in I$ for all $1\leq i\leq k$ which imply $\big< x_{s_{1}},\ldots, x_{s_{k}}\big> \subseteq (I:f)$. Since 
{\normalsize
$$X_{B^{\prime}}=\mathrm{lcm}\,\{\dfrac{u_{1}(\mathrm{pol})}{x_{s_{1},a_{s_{1}}}},\ldots, \dfrac{u_{k}(\mathrm{pol})}{x_{s_{k},a_{s_{k}}}}\}\,\, \text{and}\,\, \dfrac{u_{i}(\mathrm{pol})}{\mathrm{gcd}\,(u_{i}(\mathrm{pol}), X_{B^{\prime}})}=x_{s_{i},a_{s_{i}}},$$
}
we have $\dfrac{u_{i}}{\mathrm{gcd}\,(u_{i},f)}=x_{s_{i}}$ for all $1\leq i \leq k$. Let $u$ be a minimal generator of $I$ other than $u_{1},\ldots,u_{k}$. If $\dfrac{u}{\mathrm{gcd}\,(u,f)}\not\in \big< x_{s_{1}},\ldots, x_{s_{k}}\big>$, then by Lemma \ref{vlemma} there exists an associated prime ideal of $I$ properly containing $\big< x_{s_{1}},\ldots, x_{s_{k}}\big>$. This gives a contradiction to our assumption and so $\dfrac{u}{\mathrm{gcd}\,(u,f)}\in \big< x_{s_{1}},\ldots, x_{s_{k}}\big>$. Hence $(I:f)=\big< x_{s_{1}},\ldots, x_{s_{k}}\big>$ and $\mathrm{deg}(f)=\vert B^{\prime}\vert\leq \vert B\vert=\vert A\vert$. To find $\mathrm{min}\,D$ we can choose $M= X_{A}$ for some stable set $A$ in $\mathcal{C}$ and this completes the proof.
\end{proof}

\begin{theorem}\label{vpol}
Let $I$ be a monomial ideal. If there exists 
$$\mathfrak{p}=\big< x_{s_{1},b_{s_{1}}},\ldots, x_{s_{k},b_{s_{k}}}\big> \in \mathrm{Ass}(I(\mathrm{pol}))$$ such that $\mathrm{v}(I(\mathrm{pol}))=\mathrm{v}_{\mathfrak{p}}(I(\mathrm{pol}))$ and there is no embedded prime of $I$ properly containing $\big< x_{s_{1}},\ldots, x_{s_{k}}\big>$, then $$\mathrm{v}(I)=\mathrm{v}(I(\mathrm{pol})).$$
\end{theorem}

\begin{proof}
Let $\mathfrak{p}^{\prime}=\{x_{t_{1}},\ldots,x_{t_{r}}\}\in \mathrm{Ass}(I)$ and $f$ be the monomial such that 
$$(I:f)=\mathfrak{p}^{\prime}\,\, \text{with}\,\, \mathrm{deg}(f)=\mathrm{v}(I).$$
Then power of any $x_{i}$ in $f$ is less or equal to the highest power of $x_{i}$ appears in $G(I)$. Then by Proposition \ref{v3.1} we have
 $$(I(\mathrm{pol}):f(\mathrm{pol}))=\big< x_{t_{1},b_{t_{1}}},\ldots, x_{t_{r},b_{t_{r}}}\big>\in \mathrm{Ass}(I(\mathrm{pol})),$$ where $b_{t_{i}}-1$ is the power of $x_{t_{i}}$ in $f$ for each $1\leq i\leq r$. Thus, we have 
 $$\mathrm{v}(I(\mathrm{pol}))\leq \mathrm{v}(I),$$
 as $\mathrm{deg}(f)=\mathrm{deg}(f(\mathrm{pol}))$. Again there exists $\mathfrak{p}\in \mathrm{Ass}(I(\mathrm{pol}))$ and a square-free monomial $M$ such that
 $$(I(\mathrm{pol}):M)=\mathfrak{p}\,\, \text{with}\,\, \mathrm{deg}(M)=\mathrm{v}(I(\mathrm{pol})).$$
Then by Proposition \ref{v3.2} there exists a monomial $g$ such that 
$$(I:g)=\{x_{s_{1}},\ldots,x_{s_{k}}\}\in \mathrm{Ass}(I)\,\, \text{with}\,\, \mathrm{deg}(g)\leq \mathrm{deg}(M).$$
So we get $\mathrm{v}(I)\leq \mathrm{v}(I(\mathrm{pol}))$ and hence $\mathrm{v}(I)=\mathrm{v}(I(\mathrm{pol})).$
\end{proof}

\begin{corollary}\label{vcor3.5}
For a monomial ideal $I$, we have $\mathrm{v}(I(\mathrm{pol}))\leq \mathrm{v}(I)$. Moreover, if $I$ has no embedded prime, then $\mathrm{v}(I(\mathrm{pol}))=\mathrm{v}(I)$.
\end{corollary}

The converse of the above Corollary \ref{vcor3.5} is not necessarily true i.e., despite of having an embedded prime of a monomial ideal $I$, it may happen that $\mathrm{v}(I(\mathrm{pol}))=\mathrm{v}(I)$.

\begin{example}{\rm
Let $I=\big<x_{1}x_{2}^{2}, x_{2}x_{3}^2, x_{1}^2x_{3}\big>\subset \mathbb{Q}[x_{1},x_{2},x_{3}]$. Then 
$$I=\big<x_{2}^2,x_{3}\big>\cap  \big<x_{1},x_{3}^2\big>\cap\big<x_{1}^2,x_{2}\big>\cap\big<x_{1}^2,x_{2}^2,x_{3}^2\big>.$$ 
Here $\mathrm{Ass}(I)=\{\big<x_{2},x_{3}\big>, \big<x_{1},x_{3}\big>, \big<x_{1},x_{2}\big>, \big<x_{1},x_{2},x_{3}\big>\}$. With the help of (\cite{grv}, Procedure A1), we obtain $\mathrm{v}(I)=3=\mathrm{v}(I(\mathrm{pol}))$. In fact, we have $(I:x_{1}x_{2}x_{3})=\big<x_{1}, x_{2},x_{3}\big>$, where $\big<x_{1}, x_{2},x_{3}\big>$ is an embedded prime of $I$. Also, we have $I(\mathrm{pol})=\big< x_{1,1}x_{2,1}x_{2,2}, x_{2,1}x_{3,1}x_{3,2}, x_{3,1}x_{1,1}x_{1,2}\big>$ and 
$$(I(\mathrm{pol}):x_{1,1}x_{2,1}x_{3,1})=\big< x_{1,2},x_{2,2},x_{3,2}\big>.$$
Note that $v(I(\mathrm{pol}))=3=\mathrm{deg}(x_{1,1}x_{2,1}x_{3,1})$ and this justify our Theorem \ref{vpol} as there is no associated prime ideals of $I$ properly containing $\big<x_{1}, x_{2},x_{3}\big>$.
}
\end{example}

From Theorem \ref{vpol}, we get that the relations between $\mathrm{v}$-number of an arbitrary monomial ideal with the $\mathrm{v}$-number of its polarization. 
The next result is the generalization of (\cite{v}, Proposition 3.1) for 
a monomial ideal with some special properties.
\medskip

For a graded module $M\not= 0$, we define $\alpha(M):= \mathrm{min}\{\mathrm{deg}(f)\mid f\in M\setminus \{0\}\}$

\begin{proposition}\label{v3.4}
Let $I$ be a monomial ideal. Suppose there exists 
$$\mathfrak{p}=\big< x_{s_{1},b_{s_{1}}},\ldots, x_{s_{k},b_{s_{k}}}\big> \in \mathrm{Ass}(I(\mathrm{pol}))$$ such that $\mathrm{v}(I(\mathrm{pol}))=\mathrm{v}_{\mathfrak{p}}(I(\mathrm{pol}))$ and there is no embedded prime of $I$ properly containing $\big< x_{s_{1}},\ldots, x_{s_{k}}\big>$. Then we have
$$ \mathrm{v}(I)=\mathrm{min}\{\alpha((I:\mathfrak{p})/I)\mid \mathfrak{p}\in \mathrm{Ass}(I)\}.$$
\end{proposition}

\begin{proof}
If $I$ is a prime ideal, then $(I:1)=I$, $(I:I)=R$ and therefore we have 
$$\mathrm{v}(I)=\alpha((I:I)/I)=\alpha(R/I)=0.$$
So we may assume $I$ is not prime. Now there exists $\mathfrak{p^{\prime}}\in \mathrm{Ass}(I)$ and $f\in R_{d}$ such that $(I:f)=\mathfrak{p^{\prime}}$ with $\mathrm{v}(I)=\mathrm{deg}(f)$.  Then $f\in (I:\mathfrak{p^{\prime}})$ but $f\not\in I$ and so $f\in (I:\mathfrak{p^{\prime}})\setminus I$. Thus
$$\mathrm{v}(I)\geq \alpha((I:\mathfrak{p^{\prime}})/I)\geq \mathrm{min}\{\alpha((I:\mathfrak{p})/I)\mid \mathfrak{p}\in \mathrm{Ass}(I)\}.$$
Let us assume $\mathfrak{p}^{\prime\prime}=\big<x_{s_{1}},\ldots,x_{s_{k}}\big>\in \mathrm{Ass}(I)$ and $h\in (I:\mathfrak{p}^{\prime\prime})\setminus I$ be a monomial. Then $hx_{s_{i}}\in I$ implies $u_{i}\mid hx_{s_{i}}$ for some $u_{i}\in G(I)$, where $1\leq i\leq k$. Now $u_{i}\nmid h$ as $h\not\in I$ and so $x_{s_{i}}\mid u_{i}$. We take $h^{\prime}=\mathrm{lcm}\{\dfrac{u_{1}}{x_{s_{1}}},\ldots,\dfrac{u_{k}}{x_{s_{k}}}\}$. Then $x_{s_{i}}h^{\prime}\in I$ for all $1\leq i\leq k$ and $\mathrm{deg}(h^{\prime})\leq \mathrm{deg}(h)$. Each $\dfrac{u_{i}}{x_{s_{i}}}\mid h$ implies $h^{\prime}$ divide $h$ and hence $h^{\prime}\not\in I$ as $h\not\in I$. Therefore we can say $h^{\prime}(\mathrm{pol})\in R(\mathrm{pol})$ and since $\dfrac{u_{i}}{\mathrm{gcd}(u_{i},h^{\prime})}=x_{s_{i}}$ for each $i\in\{1,\ldots,k\}$, we also have $(h^{\prime}x_{s_{1}}\cdots x_{s_{k}})(\mathrm{pol})\in R(\mathrm{pol})$. Now consider $g=\dfrac{(h^{\prime}x_{s_{1}}\cdots x_{s_{k}})(\mathrm{pol})}{x_{s_{1},1}\cdots x_{s_{k},1}}$. Then $g\not\in I(\mathrm{pol})$ as $h^{\prime}\not\in I$. Again $x_{s_{i}}h^{\prime}\in\big< u_{i}\big>$ implies $gx_{s_{i},1}\in\big<u_{i}(\mathrm{pol})\big>$ for all $i\in\{1,\ldots,k\}$. Therefore $ g\in (I(\mathrm{pol}):\big<x_{s_{1},1},\ldots,x_{s_{k},1}\big>)\setminus I(\mathrm{pol})$ and from (\cite{far}, Proposition 2.5), we also have 
$$\big<x_{s_{1},1},\ldots,x_{s_{k},1}\big>\in \mathrm{Ass}(I(\mathrm{pol})).$$ Note that $\mathrm{deg}(g)=\mathrm{deg}(h^{\prime})\leq \mathrm{deg}(h)$. Since $h\in (I:\mathfrak{p}^{\prime\prime})\setminus I$ is arbitrary, we have
$$\alpha((I(\mathrm{pol}):\big<x_{s_{1},1},\ldots,x_{s_{k},1}\big>)/ I(\mathrm{pol}))\leq \alpha((I:\mathfrak{p}^{\prime\prime})/I).$$
 Therefore using (\cite{v}, Proposition 3.1) we get $\mathrm{v}(I(\mathrm{pol}))\leq \mathrm{min}\{\alpha((I:\mathfrak{p})/I)\mid \mathfrak{p}\in \mathrm{Ass}(I)\}$ and also by Theorem \ref{vpol} $\mathrm{v}(I)=\mathrm{v}(I(\mathrm{pol}))$. Hence $\mathrm{v}(I)\leq \mathrm{min}\{\alpha((I:\mathfrak{p})/I)\mid \mathfrak{p}\in \mathrm{Ass}(I)\}$ and the result follows.
\end{proof}

\begin{lemma}\label{vassp}
Let $I_{1}\subset R_{1}=K[\mathbf{x}]$ and $I_{2}\subset R_{2}=K[\mathbf{y}]$ be two monomial ideals and $K$ be a field. Consider $R=K[\mathbf{x},\mathbf{y}]$ and $I=I_{1}R+I_{2}R$. Then $\mathfrak{p}\in\mathrm{Ass}(R/I)$ if and only if $\mathfrak{p}=\mathfrak{p}_{1}R+\mathfrak{p}_{2}R$, where $\mathfrak{p}_{1}\in\mathrm{Ass}(R_{1}/I_{1})$ and $\mathfrak{p}_{2}\in\mathrm{Ass}(R_{2}/I_{2})$
\end{lemma}

\begin{proof}
Since $I$ is the smallest ideal containing $I_{1}R$ and $I_{2}R$, we have $$G(I)=G(I_{1})\sqcup G(I_{2}).$$ Let $\mathfrak{p}=\big< x_{s_{1}},\ldots,x_{s_{k}},y_{t_{1}},\ldots,y_{t_{l}}\big>\in\mathrm{Ass}(R/I)$. Then there exists a monomial $f\in R$ such that $(I:f)=\mathfrak{p}$. We can write $f=f_{1}f_{2}$, where $f_{1}\in R_{1}$ and $f_{2}\in R_{2}$. Now 
$$\mathfrak{p}=(I:f)=\big<\dfrac{u}{\mathrm{gcd}(u,f)}, \dfrac{v}{\mathrm{gcd}(v,f)} \mid u\in G(I_{1}),\,\, v\in G(I_{2})\big>.$$
Also, we have
 $$\dfrac{u}{\mathrm{gcd}(u,f)}=\dfrac{u}{\mathrm{gcd}(u,f_{1})}\in \mathfrak{p}\cap R_{1}\,\, \text{and} \,\,\dfrac{v}{\mathrm{gcd}(v,f)}=\dfrac{v}{\mathrm{gcd}(v,f_{2})}\in \mathfrak{p}\cap R_{2}$$
 for all $u\in G(I_{1})$ and $v\in G(I_{2})$. Therefore we get $$(I_{1}:f_{1})=\big<x_{s_{1}},\ldots,x_{s_{k}}\big>=\mathfrak{p}_{1}\in \mathrm{Ass}(R_{1}/I_{1})$$ and 
 $$(I_{2}:f_{2})=\big<y_{t_{1}},\ldots,y_{t_{l}}\big>=\mathfrak{p}_{2}\in \mathrm{Ass}(R_{2}/I_{2}).$$ 
 Hence $\mathfrak{p}=\mathfrak{p}_{1}R+\mathfrak{p}_{2}R$, where $\mathfrak{p}_{1}\in\mathrm{Ass}(R_{1}/I_{1})$ and $\mathfrak{p}_{2}\in\mathrm{Ass}(R_{2}/I_{2})$.

Again let $\mathfrak{p}=\mathfrak{p}_{1}R+\mathfrak{p}_{2}R$, where $\mathfrak{p}_{1}\in\mathrm{Ass}(R_{1}/I_{1})$ and $\mathfrak{p}_{2}\in\mathrm{Ass}(R_{2}/I_{2})$. Then clearly $\mathfrak{p}$ is a prime ideal in $R$ containing $I$. We have monomials $f_{1}\in R_{1}$ and $f_{2}\in R_{2}$ such that $(I_{1}:f_{1})=\mathfrak{p}_{1}$ and $(I_{2}:f_{2})=\mathfrak{p}_{2}$. Setting $f=f_{1}f_{2}$, we get for all $u\in G(I_{1})$ and $v\in G(I_{2})$, 
$$\dfrac{u}{\mathrm{gcd}(u,f)}=\dfrac{u}{\mathrm{gcd}(u,f_{1})}\in \mathfrak{p}_{1}\,\, \text{and} \,\,\dfrac{v}{\mathrm{gcd}(v,f)}=\dfrac{v}{\mathrm{gcd}(v,f_{2})}\in \mathfrak{p}_{2}.$$
As $(I:f)=\big<\dfrac{w}{\mathrm{gcd}(w,f)}\mid w\in G(I)\big>$ and $G(I)=G(I_{1})\sqcup G(I_{2})$, we have $(I:f)=\mathfrak{p}$ i.e., $\mathfrak{p}\in \mathrm{Ass}(R/I)$.

\end{proof}

In (\cite{v}, Proposition 3.8), the additivity of the 
$\mathrm{v}$-number for square-free monomial ideals was shown. 
In the next proposition, we show that the $\mathrm{v}$-number 
is additive for arbitrary monomial ideals.

\begin{proposition}[v-number is additive]\label{v3.5}
Let $I_{1}\subset R_{1}=K[\mathbf{x}]$ and $I_{2}\subset R_{2}=K[\mathbf{y}]$ be two monomial ideals and consider $R=K[\mathbf{x,y}]$. Then we have
$$ \mathrm{v}(I_{1}R + I_{2}R) = \mathrm{v}(I_{1}) + \mathrm{v}(I_{2}).$$
\end{proposition}

\begin{proof}
Let $I=I_{1}R+I_{2}R$. Then there exists a monomial $f\in R$ and $\mathfrak{p}\in \mathrm{Ass}(R/I)$ such that 
$$ (I:f)=\mathfrak{p}\,\, \text{and}\,\, \mathrm{v}(I)=\mathrm{deg}(f).$$
We can write $f=f_{1}f_{2}$ such that $f_{1}\in R_{1}$ and $f_{2}\in R_{2}$. Then by Lemma \ref{vassp}, we have $\mathfrak{p}=\mathfrak{p}_{1}R+\mathfrak{p}_{2}R$, where 
$$(I_{1}:f_{1})=\mathfrak{p}_{1}\in \mathrm{Ass}(R_{1}/I_{1})\,\, \text{and}\,\, (I_{2}:f_{2})=\mathfrak{p}_{2}\in\mathrm{Ass}(R_{2}/I_{2}).$$
By definition of v-number, $\mathrm{v}(I_{1})+\mathrm{v}(I_{2})\leq \mathrm{deg}(f_{1})+\mathrm{deg}(f_{2})=\mathrm{v}(I)$. For the reverse inequality, we choose monomials $f_{i}\in R_{i}$ and $\mathfrak{p}_{i}\in \mathrm{Ass}(R_{i}/I_{i})$ such that
$(I_{i}:f_{i})=\mathfrak{p}_{i}$ and $\mathrm{v}(I_{i})=\mathrm{deg}(f_{i})$, where $i\in\{1,2\}$. Again by Lemma \ref{vassp}, we have $\mathfrak{p}=\mathfrak{p}_{1}R+\mathfrak{p}_{2}R\in \mathrm{Ass}(R/I)$ and $(I:f_{1}f_{2})=\mathfrak{p}$. Thus, $\mathrm{v}(I)\leq \mathrm{deg}(f_{1}f_{2})=\mathrm{v}(I_{1})+\mathrm{v}(I_{2})$.
\end{proof}

The next result is the generalization of (\cite{v}, Proposition 3.9).

\begin{proposition}\label{v3.6}
Let $I$ be a complete intersection monomial ideal with $G(I)=\{ \mathbf{x}^{\mathbf{a}_{1}},\ldots, \mathbf{x}^{\mathbf{a}_{k}}\}$. If $d_{i}=\mathrm{deg}(\mathbf{x}^{\mathbf{a}_{i}})$ for all $i=1,\ldots,k$, then we have
$$ \mathrm{v}(I)=d_{1}+\cdots + d_{k}-k=\mathrm{reg}(R/I).$$
\end{proposition}

\proof 
$I$ is complete intersection implies that $\vert G(I)\vert =\mathrm{ht}(I)$. By (\cite{far}, Proposition 2.3), $\mathrm{ht}(I)=\mathrm{ht}(I(\mathrm{pol}))$ and therefore we have
$$\vert G(I(\mathrm{pol}))\vert =\vert G(I)\vert =\mathrm{ht}(I)=\mathrm{ht}(I(\mathrm{pol}))$$
i.e., $I(\mathrm{pol})$ is complete intersection. According to (\cite{hhmon}, Corollary 1.6.3), $\mathrm{reg}(R/I)= \mathrm{reg}(R(\mathrm{pol})/I(\mathrm{pol}))$. Since $I$ is complete intersection, $I$ has no embedded prime. Therefore, by Theorem \ref{vpol}, we have $\mathrm{v}(I)=\mathrm{v}(I(\mathrm{pol}))$. Again $\mathrm{deg}(\mathbf{x}^{\mathbf{a}_{i}}(\mathrm{pol}))=\mathrm{deg}(\mathbf{x}^{\mathbf{a}_{i}})=d_{i}$ for $i=1,\ldots,k$ and hence by (\cite{v}, Proposition 3.9), we have
$$\mathrm{v}(I)=\mathrm{v}(I(\mathrm{pol}))=d_{1}+\cdots + d_{k}-k=\mathrm{reg}(R/I). \qed$$

\begin{proposition}\label{v3.7}
Let $I$ be a monomial ideal and $f$ be a monomial such that $f\not\in I$. Then 
$\mathrm{v}(I)\leq \mathrm{v}(I:f)+\mathrm{deg}(f).$
\end{proposition}

\proof 
Suppose $(I:f)$ is an associated prime of $I$. Then by definition of $\mathrm{v}$-number, $\mathrm{v}(I)\leq \mathrm{deg}(f)$  and so the result follows as Proposition \ref{v3.6} implies $\mathrm{v}(I:f)=0$. Now assume $(I:f)\not\in \mathrm{Ass}(I)$. Then there exists an associated prime $\mathfrak{p}$ of $(I:f)$ and a monomial $g$ such that $((I:f):g)=\mathfrak{p}$ and $\mathrm{v}(I:f)=\mathrm{deg}(g)$. Note that $(I:fg)=\mathfrak{p}$ and hence we get 
$$\mathrm{v}(I)\leq \mathrm{deg}(fg)=\mathrm{v}(I:f)+\mathrm{deg}(f). \qed$$

\begin{corollary}\label{v3.8}
Let $I$ be a monomial ideal and $x_{i}$ be a variable such that 
$x_{i}\not\in I$. Then 
$\mathrm{v}(I)\leq \mathrm{v}(I:x_{i})+1.$
\end{corollary}

\begin{proof}
The result follows by taking $f=x_{i}$ in Proposition \ref{v3.7}.
\end{proof}

Some properties of v-number of edge ideals of graphs were discussed 
in (\cite{v}, proposition 3.12). We extend some of those for edge ideals of clutters, i.e., for any square-free monomial ideal in Proposition \ref{v3.9}.

\begin{proposition}\label{v3.9}
Let $I=I(\mathcal{C})$ be an edge ideal of a clutter $\mathcal{C}$.  Then the following results are true.

\begin{enumerate}[(i)]
\item If $\{x_{i}\}\not\in E(\mathcal{C})$, 
then $\mathrm{v}(I)\leq \mathrm{v}(I:x_{i})+1$, where 
$x_{i}\in V(\mathcal{C})$.
\medskip

\item $\mathrm{v}(I:x_{i})\leq \mathrm{v}(I)$, 
for some $x_{i}\in V(\mathcal{C})$.
\medskip

\item If $\mathrm{v}(I)\geq 2$, then $\mathrm{v}(I:x_{i})< \mathrm{v}(I)$ for some $x_{i}\in V(\mathcal{C})$.
\medskip

\item $\mathrm{v}(I(\mathcal{C}\setminus \{x_{i}\}))\leq \mathrm{v}(I(\mathcal{C}))$ for some $x_{i}\in V(\mathcal{C})$.
\medskip

\end{enumerate}
\end{proposition}

\begin{proof}
\noindent (i): Follows from Corollary \ref{v3.8}.
\medskip

\noindent (ii): By Lemma \ref{v2.2} and Theorem \ref{v2.3}, we have a stable set $A$ of $\mathcal{C}$ such that
$$ (I:X_{A})=\big<\mathcal{N}_{\mathcal{C}}(A)\big>=\mathfrak{p}\in \mathrm{Ass}(I)\,\,\text{and}\,\, \mathrm{v}(I)=\vert A\vert.$$
We are assuming $I\not=\mathfrak{m}$ otherwise $(\mathfrak{m}:x_{i})=R$ for any $x_{i}\in V(\mathcal{C})$. Then there exists some $x_{i}\in V(\mathcal{C})$ which is not in $\mathfrak{p}$. Note that $\mathfrak{p}\subseteq (I:x_{i}X_{A})$. Let us take $f\in (I:x_{i}X_{A})$. Then $fx_{i}\in\mathfrak{p}$ and $x_{i}\not\in\mathfrak{p}$ together imply $f\in\mathfrak{p}$. Thus $(I:x_{i}X_{A})=\mathfrak{p}$ i.e., $((I:x_{i}):X_{A})=\mathfrak{p}$. Therefore we have 
$$\mathrm{v}(I:x_{i})\leq \vert A\vert=\mathrm{v}(I).$$

\noindent (iii): Take a stable set $A$ of $\mathcal{C}$ with
$$ (I:X_{A})=\big<\mathcal{N}_{\mathcal{C}}(A)\big>=\mathfrak{p}\in \mathrm{Ass}(I)\,\,\text{and}\,\, \mathrm{v}(I)=\vert A\vert.$$
Since $\vert A \vert\geq 2$, we have $A^{\prime}=A\setminus \{x_{i}\}\not=\phi$ for any $x_{i}\in A$. Then 
$$(I:X_{A})=(I:x_{i}X_{A^{\prime}})=((I:x_{i}):X_{A^{\prime}})=\mathfrak{p},$$ which gives $\mathrm{v}(I:x_{i})\leq \vert A^{\prime}\vert <\vert A\vert=\mathrm{v}(I)$.
\medskip

\noindent (iv): Note that if $I=I(\mathcal{C})$ then $\mathrm{v}(I,x_{i})=\mathrm{v}(I(\mathcal{C}\setminus \{x_{i}\}))$. Take $A$ and $\mathfrak{p}$ as in part (ii). Pick $x_{i}\in V(\mathcal{C})\setminus A$ and so $A$ is a stable set of the clutter $\mathcal{C}\setminus \{x_{i}\}$ also. Let $e\in E(\mathcal{C}\setminus \{x_{i}\})\subset E(\mathcal{C})$. Then there exists $y\in \mathcal{N}_{\mathcal{C}}(A)$ such that $y\in e$. Also by definition of $\mathcal{N}_{\mathcal{C}}(A)$ there exists $e^{\prime}\in E(C)$ such that $e^{\prime}\subseteq A\cup \{y\}$. Now $x_{i}\not\in e$ implies $y\not = x_{i}$ and therefore $x_{i}\not\in e^{\prime}$. Then we have $e^{\prime}\in E(\mathcal{C}\setminus \{x_{i}\})$ which imply $y\in \mathcal{N}_{\mathcal{C}\setminus \{x_{i}\}}(A)$. Thus $\mathcal{N}_{\mathcal{C}\setminus \{x_{i}\}}(A)$ is a vertex cover of $\mathcal{C}\setminus \{x_{i}\}$ and $A$ being a stable set of $\mathcal{C}\setminus \{x_{i}\}$, using Lemma \ref{v2.2}, we have
$$(I(\mathcal{C}\setminus \{x_{i}\}):X_{A})=\big< \mathcal{N}_{\mathcal{C}\setminus \{x_{i}\}}(A)\big>.$$
Indeed, it is easy to see that $\mathcal{N}_{\mathcal{C}\setminus \{x_{i}\}}(A)=\mathcal{N}_{\mathcal{C}}(A)\setminus \{x_{i}\}$. Hence by Theorem \ref{v2.3}, we get $\mathrm{v}(I,x_{i})=\mathrm{v}(I(\mathcal{C}\setminus \{x_{i}\}))\leq \vert A\vert= \mathrm{v}(I).$ 
\end{proof}

\begin{proposition}\label{v3.10}
For a graph $G$, we have $\mathrm{v}(I(G))\leq \alpha_{0}(G)$.
\end{proposition}

\begin{proof}
Let $A$ be a minimal vertex cover of $G$ with $\vert A\vert=\alpha_{0}(G)$. Since $A$ is a minimal vertex cover for $G$, for each $x\in A$ there exists an edge $e_{x}\in E(G)$ which is not adjacent to any other vertex of $A$ i.e., $e_{x}\cap A=\{x\}$. Let $e_{x}=\{x,y_{x}\}$ for every $x\in A$ and $B=\{y_{x}\mid x\in A\}$. For different $x\in A$ some $y_{x}$ may coincide and so $\vert B\vert\leq \vert A\vert=\alpha_{0}(G)$. By our choice of $B$, it is clear that $A\cap B=\phi$ and so $B$ is a stable set in $G$. Also we have $\mathcal{N}_{G}(B)=A$ and hence by Lemma \ref{v2.2}, $(I(G):X_{B})=\big<\mathcal{N}_{G}(B)\big>$. 
Thus, Theorem \ref{v2.3} gives $\, \mathrm{v}(I(G))\leq \vert B\vert \leq \alpha_{0}(G)$. 
\end{proof}

\section{Bound of Regularity and Induced Matching Number by the $\mathrm{v}$-Number}
The \textit{line} graph of a graph $G$, denoted by $L(G)$, is a graph on the vertex set $V(L(G))=E(G)$ and the edge set 
$$E(L(G))=\{\{e_{i},e_{j}\}\subset E(G)\mid e_{i}\cap e_{j}\not=\phi\,\, \text{in}\,\, G\}.$$ For a positive integer $k$, the \textit{$k$-th power} of $G$, denoted by $G^{k}$, is the graph on the vertex set $V(G^{k})=V(G)$ such that there is an edge between two vertices of $G^{k}$ if and only if distance between the corresponding vertices in $G$ is less than or equal to $k$. 
\medskip

Finding a matching in a graph $G$ is equivalent to find an independent set in $L(G)$ (see \cite{bls}) and an induced matching in $G$ is equivalent to an independent set in $L^{2}(G)$, the square of $L(G)$ (see \cite{cam}). 
Now we want to know the relation between $\mathrm{v}(I(G))$ and 
$\mathrm{im}(G)$, which might be a step forward towards answering the 
Question \ref{v5.2}. In the next proposition, we try to see $\mathrm{v}(I(G))$ 
in terms of some invariant in the graph $L(G)$. What is remaining is 
to see the connection between $\mathrm{v}(I(G))$ with invariants of the 
graph $L^{2}(G)$.

\begin{proposition}\label{v4.1}
Let $G$ be a simple graph and $L(G)$ be its line graph. Suppose 
that $c(L(G))$ denotes the minimum number of cliques in $L(G)$, 
such that any vertex of $L(G)$ is either a vertex of those 
cliques or adjacent to some vertices of those cliques. Then 
$\mathrm{v}(I(G))=c(L(G))$.
\end{proposition}

\begin{proof}
Lemma \ref{v2.2} and Theorem \ref{v2.3} ensure that there exists a stable 
set $A$ in $G$ such that
$$(I(G):X_{A})=\big<\mathcal{N}_{G}(A)\big>\,\, \text{and}\,\, \vert A\vert =\mathrm{v}(I(G)).$$
For  each $x_{i}\in A$, let $E_{G}(x_{i})=\{e_{i1},\ldots,e_{im_{i}}\}$ be the set of edges incident to the vertex $x_{i}$. Then $E_{G}(x_{i})$ forms clique in $L(G)$ for each $x_{i}\in A$. Since $A$ is stable, cliques corresponding to each $E_{G}(x_{i})$, where $x_{i}\in A$, are disjoint to each other. Let $e\in V(L(G))$ be a vertex other than the vertices of the cliques $E_{G}(x_{i})$, for $x_{i}\in A$. Let $e=\{u,v\}$ be the corresponding edge of $e$ in $G$. Then one of $u$ or $v$ should belong to $\mathcal{N}_{G}(A)$ as $\mathcal{N}_{G}(A)$ is a minimal vertex cover of $G$. Assume $u\in \mathcal{N}_{G}(A)$ and our choice of $e$ ensures that $v\not\in A$. Then $u\in \mathcal{N}_{G}(x_{i})$ and so $\{x_{i},u\}=e_{ik}$ for some $1\leq k\leq m_{i}$. Therefore $e$ and $e_{ik}$ being adjacent in $G$, we have $e$ is adjacent to the vertex $e_{ik}\in E_{G}(x_{i})$ in $L(G)$. Hence we have $c(L(G))\leq \vert A\vert= \mathrm{v}(I(G))$. 
\medskip

Now for the reverse inequality, let $r=c(L(G))$ and we can choose $r$ disjoint cliques $\mathcal{C}_{1},\ldots, \mathcal{C}_{r}$ in $L(G)$ such that any vertex of $L(G)$ is either a vertex of $\mathcal{C}_{i}$ or adjacent to some vertices of $\mathcal{C}_{i}$, $1\leq i\leq r$. Since each $\mathcal{C}_{i}$ is a clique in $L(G)$, corresponding edges in $G$ of vertices of $\mathcal{C}_{i}$, either shares a common vertex or they forms a triangle in $G$. Suppose corresponding edges of $\mathcal{C}_{1},\ldots,\mathcal{C}_{k}$ in $G$ share a common vertex, say $x_{1},\ldots,x_{k}$ respectively and corresponding edges  in $G$ of $\mathcal{C}_{k+1},\ldots,\mathcal{C}_{r}$ forms triangles. Take one vertex from each triangle formed by the corresponding edges  in $G$ of $\mathcal{C}_{k+1},\ldots,\mathcal{C}_{r}$, say $x_{k+1},\ldots,x_{r}$. Since $\mathcal{C}_{1},\ldots, \mathcal{C}_{r}$ are disjoint in $L(G)$, $B=\{x_{1}\ldots,x_{r}\}$ is a stable set in $G$. We will show $\mathcal{N}_{G}(B)$ forms a minimal vertex cover for $G$. Pick any $e=\{u,v\}\in E(G)$. Then $e\in V(L(G))$ and if $e\in \mathcal{C}_{i}$ for $1\leq i\leq r$ then one of $u$ or $v$ should belong to $\mathcal{N}_{G}(x_{i})$. Suppose $e$ is a vertex other than the vertices of $\mathcal{C}_{1},\ldots,\mathcal{C}_{r}$ in $L(G)$. Then $e\cap B=\phi$ and $e$ is adjacent to some vertex $e_{ij}\in \mathcal{C}_{i}$ in $L(G)$, $1\leq i\leq r$. Therefore $e$ and $e_{ij}$ share a common vertex, say $u$, in $G$. Then $u\in\mathcal{N}_{G}(x_{i})$ and so $\mathcal{N}_{G}(B)$ is a vertex cover for $G$. Thus using Lemma \ref{v2.2}, we get
$$(I(G):X_{B})=\big<\mathcal{N}_{G}(B)\big>$$
and hence $\mathrm{v}(I(G))\leq \vert B\vert=r=c(L(G))$.
\end{proof}

Let $G$ be a simple graph and $e\in E(G)$ be an edge.
\begin{enumerate}
\item[$\bullet$] We define $G\setminus e$ as the graph on $V(G)$ just by removing the edge $e$ from $E(G)$.
\item[$\bullet$] By $G_{e}$ we mean the induced subgraph of $G$ on the vertex set $V(G)\setminus \mathcal{N}_{G\setminus e}[e]$.
\item[$\bullet$] The contraction of $e$ on $G$ (see \cite{bcreg}, Definition 5.2), denoted by $G/e$, is defined by $V(G/e)=(V(G)\setminus e)\cup \{w\}$, where $w$ is a new vertex, and $E(G/e)=E(G\setminus e)\cup \{\{w,z\}: z\in\mathcal{N}_{G\setminus e}(e)\}$.
 \end{enumerate}
\medskip

Let is first cite some results which give ome bounds of 
$\mathrm{reg}(R/I(G))$ in terms of some graph obtained from $G$:
\begin{enumerate}
\item[(1)] From \cite{hahui}, Theorem 3.5 we get 
$$\mathrm{reg}(R/I(G))\leq \mathrm{max}\{\mathrm{reg}(R/I(G\setminus e)),\mathrm{reg}(R/I(G_{e}))+1\}.$$
\item[(2)] In \cite{bcreg}, Biyiko$\breve{\mathrm{g}}$lu and Civan proved that
$$ \mathrm{reg}(R/I(G/e))\leq \mathrm{reg}(R/I(G))\leq \mathrm{reg}(R/I(G/e))+1.$$
\item[(3)] (\cite{w}, Theorem 3). Let $J\subset V(G)$ be an induced clique in $G$. Then 
$$\mathrm{reg}(R/I(G))\leq \mathrm{reg} (R/I(G\setminus J)) + 1,$$ where $G\setminus J$ denotes the induced subgraph on $V(G)\setminus J$.
\end{enumerate}
\medskip

As a consequence of the above results, we prove the following 
Proposition \ref{v4.2}, which might be helpful in finding a 
relation between the $\mathrm{v}$-number and regularity using 
induction hypothesis.
\medskip

\begin{proposition}\label{v4.2}
Let $G$ be a simple graph. Then
\medskip

\begin{enumerate}[(i)]
\item $\mathrm{v}(I(G\setminus e))\leq \mathrm{v}(I(G))+1$, for any $e\in E(G)$.
\medskip

\item $\mathrm{v}(I(G))\leq \mathrm{v}(I(G\setminus J))+1$, where $J$ is a clique of $G$.
\medskip

\item There exists an edge $e\in E(G)$ such that $\mathrm{v}(I(G/e))\leq \mathrm{v}(I(G))$.
\medskip

\end{enumerate}  
\end{proposition}

\begin{proof}

\noindent (i): By Lemma \ref{v2.2} and Theorem \ref{v2.3}, there exists a stable set $A$ of $G$ such that 
$$ (I(G):X_{A})=\big<\mathcal{N}_{G}(A)\big>\,\, and\,\, \mathrm{v}(I(G))=\vert A\vert.$$
Clearly, $A$ is a stable set too in $G\setminus e$. 
\medskip

\noindent\textbf{Case-I:} Suppose $e\cap A=\phi$. Then $\mathcal{N}_{G\setminus e}(A)=\mathcal{N}_{G}(A)$ and it is also a vertex cover for $G\setminus e$. Thus using Lemma \ref{v2.2}, we have
$$(I(G\setminus e):X_{A})=\big<\mathcal{N}_{G\setminus e}(A)\big>,$$
which imply $\mathrm{v}(I(G\setminus e))\leq \mathrm{v}(I(G))$.
\medskip

\noindent\textbf{Case-II:} Let $e\cap A\not =\phi$ and $u\in e\cap A$, where $e=\{u,v\}$. Then $v\in \mathcal{N}_{G}(A)$. If $v\in\mathcal{N}_{G\setminus e}(A)$, then $\mathcal{N}_{G}(A)=\mathcal{N}_{G\setminus e}(A)$ is a vertex cover of $G\setminus e$. Therefore by Lemma \ref{v2.2}, 
$$(I(G\setminus e):X_{A})=\big<\mathcal{N}_{G\setminus e}(A)\big>,$$
and so $\mathrm{v}(I(G\setminus e))\leq \mathrm{v}(I(G))$. If $v\not\in\mathcal{N}_{G\setminus e}(A)$, then $A\cup \{v\}$ is a stable set in $G\setminus e$ and $\mathcal{N}_{G\setminus e}(A\cup\{v\})$ forms a vertex cover for $G\setminus e$. Again by Lemma \ref{v2.2}, we have 
$$ (I(G\setminus e):X_{A\cup\{v\}})=\big<\mathcal{N}_{G\setminus e}(A\cup\{v\})\big>.$$
Hence $\mathrm{v}(I(G\setminus e))\leq \vert A\vert +1=\mathrm{v}(I(G))+1$.
\medskip

\noindent (ii): From Lemma \ref{v2.2} and Theorem \ref{v2.3}, we have a stable set $A$ of $G\setminus J$ such that 
$$(I(G\setminus J):X_{A})=\big< \mathcal{N}_{G\setminus J}(A)\big>\,\, \text{and}\,\, \mathrm{v}(I(G\setminus J))=\vert A\vert.$$
Note that $A$ is also a stable set in $G$. If all vertices of $J$ is contained in $\mathcal{N}_{G}(A)$, then $\mathcal{N}_{G}(A)$ is a vertex cover of $G$ and by Lemma \ref{v2.2}, we have $(I(G):X_{A})=\big< \mathcal{N}_{G}(A)\big>$. Thus by Theorem \ref{v2.3}, $\mathrm{v}(I(G))\leq \mathrm{v}(I(G\setminus J))$. Suppose there is a vertex $x\in J$ such that $x\not\in \mathcal{N}_{G}(A)$. Then $A\cup\{x\}$ is a stable set in $G$ and $\mathcal{N}_{G}(A\cup\{x\})$ is a vertex cover of $G$. So by Lemma \ref{v2.2},
$$(I(G):X_{A\cup\{x\}})=\big< \mathcal{N}_{G}(A\cup\{x\})\big>.$$
Hence by Theorem \ref{v2.3}, $\mathrm{v}(I(G))\leq \mathrm{v}(I(G\setminus J))+1$.
\medskip

\noindent (iii): By Lemma \ref{v2.2} and Theorem \ref{v2.3}, there is a stable set $A$ of $G$ such that 
$$(I(G):X_{A})=\big<\mathcal{N}_{G}(A)\big>\,\, \text{and}\,\, \mathrm{v}(I(G))=\vert A\vert.$$
Let $u\in A$ and by minimality of $A$ there exists $v\in \mathcal{N}_{G}(u)$ such that $v\not\in \mathcal{N}_{G}(A\setminus\{u\})$. Contract the edge $e=\{u,v\}$ in $G$ and let after contracting $e$ we get the vertex $w$ in $G/e$ instead of $u$ and $v$. Then $B=(A\setminus \{u\})\cup\{w\}$ is a stable set in $G/e$. It is clear that $\mathcal{N}_{G/e}(B)$ is a vertex cover for $G/e$ and so using Lemma \ref{v2.2} we get 
$$(I(G/e):X_{B})=\big<\mathcal{N}_{G/e}(B)\big>.$$ 
Therefore by Theorem \ref{v2.3}, $\mathrm{v}(I(G/e))\leq \vert B\vert=\vert A\vert= \mathrm{v}(I(G))$.
\end{proof}

In \cite{lz}, Liu and Zhou gave formula for induced matching number of a graph in terms of its induced bipartite subgraph. Using that formula we show $\mathrm{v}(I(G))\leq \mathrm{im}(G)$ for any bipartite graph $G$ (see Theorem \ref{imbvr}). 
\begin{theorem}[\cite{lz}, Theorem 2.1]
For a simple graph $G$,
$$\mathrm{im}(G)=\underset{H}{\mathrm{max}}\,\mathrm{min}\{\vert X^{\prime}\vert : X^{\prime}\subseteq X\,\, \text{and}\,\, Y\subseteq \mathcal{N}_{H}(X^{\prime})\},$$
where $H$ is an induced bipartite subgraph of $G$ with partite sets $X,Y$ and has no isolated vertices.
\end{theorem}

\begin{theorem}[\cite{lz}, Theorem 2.3]\label{imb}
Let $G$ be a bipartite graph with partite sets $X,Y$ and has no isolated vertices. Then
{\fontsize{13}{14}\selectfont
 $$\mathrm{im}(G)=\underset{H}{\mathrm{max}}\,\mathrm{min}\{\vert X^{\prime}\vert : X^{\prime}\subseteq H\subseteq X\,\, \text{and}\,\, \mathcal{N}_{G}(X^{\prime})=\mathcal{N}_{G}(H)\}.$$}
\end{theorem}

\begin{theorem}\label{imbvr}
Let $G$ is a bipartite graph with partite sets $X$ and $Y$. Then $\mathrm{v}(I(G))\leq \mathrm{im}(G)$. Moreover, we have
 $$\mathrm{v}(I(G))\leq \mathrm{reg}(R/I(G)).$$
\end{theorem}

\begin{proof}
Let $X_{1}\subseteq X$ be such that $\mathcal{N}_{G}(X_{1})=\mathcal{N}_{G}(X)$ and
$$\vert X_{1}\vert=\mathrm{min}\{\vert X^{\prime}\vert : X^{\prime}\subseteq X\,\, \text{and}\,\, \mathcal{N}_{G}(X^{\prime})=\mathcal{N}_{G}(X)\}.$$
Then $X_{1}$ is a stable set in $G$ and $\mathcal{N}_{G}(X)$ being a minimal vertex cover for $G$, we have by Theorem \ref{v2.3}, $\mathrm{v}(I(G))\leq \vert X_{1}\vert$. Now taking $H=X$ in Theorem \ref{imb}, we get 
$$\mathrm{v}(I(G))\leq \vert X_{1}\vert\leq \mathrm{im}(G).$$
Therefore by (\cite{katz}, Lemma 2.2 and \cite{hahui}, Theorem 4.1),we have 
$$\mathrm{v}(I(G))\leq \mathrm{im}(G)\leq \mathrm{reg}(R/I(G)).$$
\end{proof}

\begin{corollary}\label{v4.11}
Let $G$ be a graph with a vertex $x\in V(G)$, such that any of the following holds:

\begin{enumerate}[(i)]
\item The independent complex $\Delta(G\setminus \{x\})$ or $\Delta(G\setminus \mathcal{N}_{G}[x])$ is vertex decomposable.

\item The graph $G\setminus \{x\}$ or $G\setminus \mathcal{N}_{G}[x]$ is a bipartite graph.

\end{enumerate}
Then $\mathrm{v}(I(G))\leq \mathrm{reg}(R/I(G))+1$.
\end{corollary}

\begin{proof}
Let $I=I(G)$. If the condition $\mathrm{(i)}$ or $\mathrm{(ii)}$ holds, then by  Theorem \ref{imbvr} and (\cite{v}, Theorem 3.13), we have $$\mathrm{v}(I,x)\leq \mathrm{reg}(R/(I,x))\,\, \mathrm{or} \,\, \mathrm{v}(I:x)\leq \mathrm{reg}(R/(I:x)).$$ Now by (\cite{v}, Lemma 3.12), we have 
$$\mathrm{v}(I)\leq \mathrm{v}(I,x)+1\,\, \mathrm{and}\,\, \mathrm{v}(I)\leq \mathrm{v}(I:x)+1.$$
 Also $G\setminus \{x\}$ and $G\setminus \mathcal{N}_{G}[x]$ being a subgraph of $G$, (\cite{vil}, Proposition 6.4.6) implies that $\mathrm{reg}(R/(I,x))\leq \mathrm{reg}(R/I)$ and $ \mathrm{reg}(R/(I:x))\leq \mathrm{reg}(R/I)$. Therefore for any cases we have $\mathrm{v}(I)\leq \mathrm{reg}(R/I)+1$.
\end{proof}

\begin{corollary}
If $G$ is an unicyclic graph i.e., a graph with only one induced cycle, then $\mathrm{v}(I(G))\leq \mathrm{reg}(R/I(G))+1$.
\end{corollary}

\begin{proof}
Choose a vertex $x$ from the unique induced cycle of $G$. Then $G\setminus \{x\}$ is a bipartite graph and hence, by Corollary \ref{v4.11}, the result follows.
\end{proof}

\begin{definition}[\cite{cam}]{\rm
A \textit{clique-neighbourhood} $K_{c}$ is the set of edges of a clique $c$ in a graph  $G$ together with some edges which are adjacent to some edges of the clique $c$.
}
\end{definition}

\begin{theorem}[\cite{cam}, Theorem 2]\label{v4.7}
Let $G$ be a chordal graph. Then 
\begin{align*}
\mathrm{im}(G)=\mathrm{min}\{  \vert \mathscr{N}\vert : & \mathscr{N} \,\, \text{is\,\, a\,\, set\,\, of\,\, clique-neighbourhoods}\\ & \text{in}\,\, G\,\,\text{which\,\, covers}\,\, E(G)\}.
\end{align*}
\end{theorem}

We now prove that $\mathrm{v}(I(G))\leq \mathrm{im}(G)=\mathrm{reg}(R/I(G))$ is true for chordal graphs. 
This also follows from Theorem \ref{v4.9}, 
where we prove the same inequality for a more general 
class. However, the proofs of Theorem \ref{v4.8} and 
Theorem \ref{v4.9} are of different flavour. 

\begin{theorem}\label{v4.8}
For a chordal graph $G$, we have 
$$\mathrm{v}(I(G))\leq \mathrm{im}(G)=\mathrm{reg}(R/I(G)).$$
\end{theorem}

\begin{proof}
Let $\mathscr{N}$ be a set of clique-neighbourhoods in $G$ which covers $E(G)$. Let $\mathscr{N}=\{K_{c_{1}},\ldots,K_{c_{m}}\}$, where each $K_{c_{i}}$, for $1\leq i\leq m$, is a clique-neighbourhood containing the clique $c_{i}$ such that every edge of $K_{c_{i}}$ is adjacent to some edges of $c_{i}$. Now choose a maximal stable set from the set of vertices $\bigcup_{i=1}^{m}V(c_{i})$, named it $A$. Since $A$ is maximal stable set in $\bigcup_{i=1}^{m}V(c_{i})$, we have $\bigcup_{i=1}^{m}V(c_{i})\setminus A \subset \mathcal{N}_{G}(A)$. Let $e\in E(G)$ be any edge. Then $e\in K_{c_{i}}$ for some $1\leq i\leq m$. If $e\in E(c_{i})$ then $e\cap \mathcal{N}_{G}(A)\not = \phi$. Suppose $e\not\in E(c_{i})$. Then $e$ is adjacent to some edges of $c_{i}$ i.e., $e$ is incident to some vertex $v\in V(c_{i})$. Now if $v\not \in \mathcal{N}_{G}(A)$ then $v\in A$ as $A$ is maximal stable set in  $\bigcup_{i=1}^{m}V(c_{i})$ and so $e\setminus \{v\}\in \mathcal{N}_{G}(A)$. As $e\in E(G)$ is arbitrary chosen edge, $\mathcal{N}_{G}(A)$ is a vertex cover of $G$. Therefore by Lemma \ref{v2.2}, we have 
$$ (I(G): X_{A})=\big< \mathcal{N}_{G}(A)\big> ,$$
and so Theorem \ref{v2.3} gives $\mathrm{v}(I(G))\leq \vert A\vert$. Since $A\subset \bigcup_{i=1}^{m}V(c_{i})$ is stable set and $c_{i}$s are cliques, $\vert A\vert\leq m=\vert \mathscr{N}\vert$. This is true for any set of clique-neighbourhoods in $G$ which covers $E(G)$. Hence by Theorem \ref{v4.7}, $\mathrm{v}(I(G))\leq \mathrm{im}(G)$ and by (\cite{havan}, Corollary 6.9) $\mathrm{im}(G)=\mathrm{reg}(R/I(G))$.
\end{proof}

\begin{theorem}\label{v4.9}
If $G$ is a $(C_{4},C_{5})$-free vertex-decomposable graph, then $\mathrm{v}(I(G))\leq \mathrm{im}(G)=\mathrm{reg}(I(G))$.
\end{theorem}
\begin{proof}
If $G$ is a $(C_{4},C_{5})$-free vertex-decomposable graph, then by (\cite{bc}, Theorem 24), we get $\mathrm{im}(G)=\mathrm{reg}(I(G))$ and also $G$ being a vertex decomposable graph, by (\cite{v}, Theorem 3.13), we get $\mathrm{v}(I(G))\leq \mathrm{reg}(I(G))$.
\end{proof}

Let $G$ be a graph with $V(G)=\{x_{1},\ldots,x_{n}\}$. Consider the graph $W_{G}$ by adding a new set of vertices $Y=\{y_{1},\ldots,y_{n}\}$ to $G$ and attaching the edges $\{x_{i},y_{i}\}$ to $G$ for each $1\leq i\leq n$. The graph $W_{G}$ is known as the \textit{whisker graph} of $G$ and the attached edges $\{x_{i}, y_{i}\}$ are called the \textit{whiskers}.

\begin{theorem}\label{v4.10}
Let $G$ be a simple graph and $W_{G}$ be the whisker graph of $G$. Then $\mathrm{v}(I(W_{G}))\leq \mathrm{im}(W_{G})$.
\end{theorem}

\begin{proof}
Let $A$ be a maximal stable set of $G$. Then the set of whiskers $M=\{\{x_{i},y_{i}\}\mid x_{i}\in A\}$ forms an induced matching in $W_{G}$. Therefore we have $\mathrm{im}(W_{G})\geq \vert A\vert$. Now $A$ is a stable set in $W_{G}$ too and it is clear from the construction of $W_{G}$ that $\mathcal{N}_{W_{G}}(A)$ is a vertex cover of $W_{G}$. Thus applying Lemma \ref{v2.2} and Theorem \ref{v2.3}, we get $\mathrm{v}(I(W_{G}))\leq \vert A\vert\leq \mathrm{im}(W_{G})$.
\end{proof}

\noindent Theorem \ref{v4.10} also follows from (\cite{grv}, Theorem 2 and Lemma 1).

\begin{definition}[\cite{hkm}]{\rm
Let $G$ be a simple graph on the vertex set $V(G)=\{x_{1},\ldots, x_{n}\}$, without any 
isolated vertex. For an independent set $S\subset V(G)$, the $S$-suspension of $G$, 
denoted by $G^{S}$, is the graph given by 
\begin{enumerate}
\item[$\bullet$] $V(G^{S})= V(G)\cup \{x_{n+1}\}$, where $x_{n+1}$ is a new vertex;
\item[$\bullet$] $E(G^{S})=E(G)\cup \{\{x_{i},x_{n+1}\}\mid x_{i}\not\in S\}$.
\end{enumerate}
}
\end{definition}

\begin{proposition}\label{v4.14}
Let $G$ be a simple graph and $G^{S}$ be a $S$-suspension of $G$ with respect to an independent set $S\subset V(G)$. Then $\mathrm{v}(I(G^{S}))=1$.
\end{proposition}

\begin{proof}
Take $A=\{x_{n+1}\}$. Then we have
$$\mathcal{N}_{G^{S}}(A)= V(G)\setminus S=V(G^{S})\setminus (S\cup\{x_{n+1}\}).$$
By construction of $G^{S}$, $S\cup\{x_{n+1}\}$ is an independent set of $G^{S}$ and hence $\mathcal{N}_{G^{S}}(A)$ is a vertex cover of $G^{S}$. Therefore, from Lemma \ref{v2.2} it follows that 
$$ (I(G^{S}):X_{A})=\big<\mathcal{N}_{G^{S}}(A)\big>.$$
Thus by Theorem \ref{v2.3}, we have $\mathrm{v}(I(G^{S}))= \vert A\vert=1$.
\end{proof}

\begin{corollary}\label{4.15}
Let $n$ be any positive integer. Then we have a graph $G$ such that $\mathrm{reg}(R/I(G))-\mathrm{v}(I(G))=n$, i.e, for a simple connected graph $G$, $\mathrm{reg}(R/I(G))$ can be arbitrarily larger than $\mathrm{v}(I(G))$.
\end{corollary}

\begin{proof}
We can choose a connected graph $H$ with $\mathrm{reg}(R^{\prime}/I(H))=n+1$, where $R^{\prime}=K[V(H)]$. Consider the graph $G=H^{S}$, where $S$ is a stable set of $H$. Then Proposition \ref{v4.14} gives $\mathrm{v}(I(G))=1$ and by (\cite{hkm}, Lemma 1.5), we have $\mathrm{reg}(R/I(G))=\mathrm{reg}(R^{\prime}/I(H))=n+1$, where $R=R^{\prime}[x_{n+1}]$. Therefore $\mathrm{reg}(R/I(G))-\mathrm{v}(I(G))=n$.
\end{proof}

\begin{theorem}[Terai, \cite{terai}]\label{v4.16}
Let $I$ be a square-free monomial ideal in a polynomial ring $R$. Then 
$ \mathrm{reg}(I)=\mathrm{reg}(R/I)+1=\mathrm{pd}(R/I^{\vee})$.
\end{theorem}

\begin{definition}\label{v4.17}{\rm
Let $I=I(\mathcal{C})$ be an edge ideal of a clutter $\mathcal{C}$. The \textit{Alexander dual ideal} 
of $I$, denoted by $I^{\vee}$, is the ideal defined by
$$I^{\vee}=\big< \{X_{C}\mid C\,\, \text{is\,\, a\,\, minimal\,\, vertex\,\, cover\,\, of}\,\,\mathcal{C}\}\big>.$$
}
\end{definition}

\noindent From (\cite{v}, Lemma 3.16), we have $\mathrm{v}(I(\mathcal{C})^{\vee})\geq \alpha_{0}(\mathcal{C})-1$.

\begin{proposition}\label{v4.18}
Let $\mathcal{C}$ be a clutter that can not be written as union of two disjoint clutters. If the answer to Question \ref{v5.2} (see Section 5) is true and $\mathrm{v}(I^{\vee})\geq \alpha_{0}(\mathcal{C})+1$, then $I=I(\mathcal{C})$ is not Cohen-Macaulay.
\end{proposition}

\begin{proof}
By the Auslander-Buchsbaum formula (\cite{peeva}, Formula 15.3), we have 
$$\mathrm{depth}(R/I)+\mathrm{pd}(R/I)=n.$$
By the given condition, Question \ref{v5.2} implies $\mathrm{v}(I^{\vee})\leq \mathrm{reg}(R/I)+1$. Therefore, using Theorem \ref{v4.16} we get 
\begin{align*}
\mathrm{depth}(R/I)&=  n-1-\mathrm{reg}(R/I^{\vee})\\
& \leq  n-\mathrm{v}(I^{\vee})\\
& \leq n-1-\alpha_{0}(\mathcal{C})=\mathrm{dim}(R/I)-1.
\end{align*}
Hence $I$ is not Cohen-Macaulay as $\mathrm{depth}(R/I)<\mathrm{dim}(R/I)$.
\end{proof}

For a large class of square-free monomial ideals $I$, we have $\mathrm{v}(I)\leq \mathrm{reg}(R/I)$. 
The following Corollary gives a sufficient condition for the non Cohen-Macaulayness of 
$I = I(\mathcal{C})$.

\begin{corollary}
If $\alpha_{0}(\mathcal{C})\leq \mathrm{v}(I^{\vee})\leq \mathrm{reg}(R/I^{\vee})$, then $I=I(\mathcal{C})$ is not Cohen-Macaulay.
\end{corollary}
\begin{proof}
Follows directly from the proof of Proposition \ref{v4.18}.
\end{proof}
\medskip

\section{Some Open Problems on v-Number}
Jaramillo and Villarreal disproved (\cite{npv}, Conjecture 4.2) by giving an example (see \cite{v}, 
Example 5.4) of a graph $G$ for which $\mathrm{v}(I(G))> \mathrm{reg}(R/I(G))$. They also proposed 
an open problem, whether $\mathrm{v}(I)\leq \mathrm{reg}(R/I)+1$ for any square-free monomial ideal 
$I$. The answer is no and we give the following example in support:
\medskip

\begin{example}\label{v5.1}{\rm
Take $H=G_{1}\sqcup G_{2}$ with $G_{1}\simeq G_{2}\simeq G$, where $G$ is the graph in (\cite{v}, Example 5.4). It was given in (\cite{v}, Example 5.4) that $\mathrm{v}(I(G))=3$ and $\mathrm{reg}(R^{\prime}/I(G))=2$, where $R^{\prime}=\mathbb{Q}[V(G)]$. Then by Proposition \ref{v3.5} and by (\cite{w}, Lemma 7)  we have $\mathrm{v}(I(H))= 6$ and $\mathrm{reg}(R/I(H))=4$, where $R=\mathbb{Q}[V(H)]$. Hence $\mathrm{v}(I(H))> \mathrm{reg}(R/I(H))+1$.
}
\end{example}
\medskip

In our example the graph $H$ is not connected. So we can modify the open problem by putting the 
condition of connectedness:

\begin{question}\label{v5.2}
Let $\mathcal{C}$ be a clutter which can not be written as an union of two disjoint clutter. Then is it true that 
$$\mathrm{v}(I(\mathcal{C}))\leq \mathrm{reg}(R/I(\mathcal{C}))+1 ?$$ 
This question for graphs would be the following:

\noindent Let $G$ be a simple connected graph. Is it true that $$\mathrm{v}(I(G))\leq \mathrm{reg}(R/I(G))+1 ?$$
\end{question}

For a simple graph $G$, we have from (\cite{katz}, Lemma 2.2; \cite{hahui}, Theorem 4.1) that 
$\mathrm{im}(G)\leq \mathrm{reg}(R/I(G))$. So, we want to find a relation between 
$\mathrm{v}(I(G))$ and $\mathrm{im}(G)$ for connected graphs $G$, which might help 
is find an answer to the Question \ref{v5.2} for connected graphs. In many cases, 
we have $\mathrm{v}(I(G))\leq \mathrm{im}(G)$, for example, if $G$ is a bipartite graph (see Theorem \ref{imbvr} or if $G$ is a $(C_{4},C_{5})$-free vertex decomposable graph (see Theorem \ref{v4.9}) or $G$ is a whisker graph (see Theorem \ref{v4.10}). Let us consider the following example:
\medskip

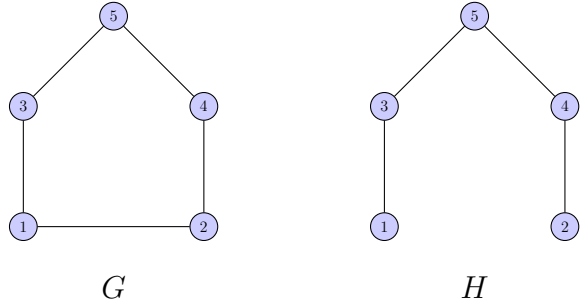
\begin{center}
\begin{tikzpicture}
  [scale=.4,auto=left,every node/.style={circle,scale=0.5}]
 
  \node[draw,fill=blue!20] (n1) at (0,0)  {$1$};
  \node[draw,fill=blue!20] (n2) at (6,0)  {$2$};
  \node[draw,fill=blue!20] (n3) at (0,4) {$3$};
   \node[draw,fill=blue!20] (n4) at (6,4) {$4$};
   \node[draw,fill=blue!20] (n5) at (3,7) {$5$};
 
\node[scale=2] (n6) at (3,-2){$G$};

\node[draw,fill=blue!20] (m1) at (12,0)  {$1$};
  \node[draw,fill=blue!20] (m2) at (18,0)  {$2$};
  \node[draw,fill=blue!20] (m3) at (12,4) {$3$};
   \node[draw,fill=blue!20] (m4) at (18,4) {$4$};
   \node[draw,fill=blue!20] (m5) at (15,7) {$5$};
 
 \node[scale=2] (n6) at (15,-2){$H$};
 
  \foreach \from/\to in {n1/n2,n1/n3,n2/n4,n3/n5, n4/n5, m1/m3,m2/m4,m3/m5, m4/m5}
    \draw[] (\from) -- (\to);
   
\end{tikzpicture}
\end{center} 
\medskip
We have $\mathrm{v}(I(G))=2$, $\mathrm{im}(G)=1$ and $\mathrm{v}(I(H))=1$, $\mathrm{im}(H)=2$. 
In view of this, we can ask the followign question, which can answer Question \ref{v5.2} for 
edge ideals of graphs.
 
 \begin{question}\label{v5.3}
 For a connected graph $G$, is it true that
 $$\mathrm{v}(I(G))\leq \mathrm{im}(G)+1?$$
 \end{question}

Moreover, we can generalize Question \ref{v5.3} for any edge ideal of a clutter $\mathcal{C}$, which can not be written as an union of two disjoint clutters (see Question \ref{v5.4}).
\medskip

Let $\mathcal{C}$ be a clutter. A set $M\subset E(\mathcal{C})$ is called a \textit{matching} in 
$\mathcal{C}$ if the edges in $M$ are pairwise disjoint. The matching $M$ is called an 
\textit{induced matching} in $\mathcal{C}$ if the induced subclutter on the vertex set $(\bigcup_{e\in M}e)$ contains only $M$ as the edge set. The maximum size of an induced matching in $\mathcal{C}$ is known as the \textit{induced matching number} of $\mathcal{C}$, denoted by $\mathrm{im}(\mathcal{C})$.
\medskip

Let $\mathcal{C}$ be a clutter and let $\{e_{1},\ldots,e_{k}\}$ form an induced matching in $\mathcal{C}$. Then (\cite{mvil}, Corollary 3.9; \cite{hahui}, Theorem 4.2) gives
$$\sum_{i=1}^{k}(\vert e_{i}\vert-1)\leq \mathrm{reg}(R/I(\mathcal{C})).$$

\begin{question}\label{v5.4}
Let $\mathcal{C}$ be a clutter which can not be written as an union of two disjoint clutters. Does there exist an induced matching $\{ e_{1},\ldots,e_{k}\}$ of $\mathcal{C}$ such that
$$ \mathrm{v}(I(\mathcal{C}))\leq \sum_{i=1}^{k}(\vert e_{i}\vert -1)+1?$$
\end{question}

An answer to Question \ref{v5.4}, together with (\cite{mvil}, Corollary 3.9; \cite{hahui}, Theorem 4.2) 
can give an answer to Question \ref{v5.2}.
\medskip

The next problem is about our interest to know the relation between $\mathrm{depth}(R/I)$ and $\mathrm{v}(I)$ for any square-free monomial ideal. If $R/I$ is Cohen-Macaulay, then by Theorem \ref{v2.3}, $\mathrm{v}(I)\leq \mathrm{depth}(R/I)$.
\medskip

\begin{question}
For a square-free monomial ideal $I$, does $\mathrm{v}(I)\leq \mathrm{depth}(R/I)$ hold? Also can we say that 
$$\mathrm{v}(I)\geq \mathrm{dim}(R/I)-\mathrm{depth}(R/I)?$$
\end{question} 
\medskip

If we can relate $\mathrm{v}(I(G))$ with respect to some invariants of $L^{2}(G)$, then it would be easy to answer Question \ref{v5.3} because $\mathrm{im}(G)=\beta_{0}(L^{2}(G))$.
 
\begin{question}
Find $\mathrm{v}(I(G))$ in terms of some invariants of $L^{2}(G)$, where $G$ is a connected graph.
\end{question}

\section*{acknowledgements}
The authors would like to record their sincere 
gratitude to Prof R.H. Villarreal for his 
valuable comments. The computer algebra software 
Macaulay2 \cite{mac2} has been used extensively for 
carrying out computations.

\section*{Data availability statement}
There is no data associated with this work. 

\bibliographystyle{amsalpha}

\end{document}